\documentclass[12pt]{amsart}
1
\usepackage{tikz}
\usepackage{mathrsfs, amssymb}
\usepackage[margin=1.995cm]{geometry}
\usepackage{hyperref}
\usepackage{array}
\usepackage{upgreek}
\usepackage[matrix,arrow,curve]{xy}
\usepackage{longtable}
\usepackage{amsthm}
\usepackage{float}

\makeatletter
\@addtoreset{equation}{subsection}
\makeatother

\newenvironment{myequation}{\begin{equation}}{\end{equation}}

\newcounter{N}
\renewcommand{\theN}{\rm\alph{N})}
\newcommand{\nN}{\rm\refstepcounter{N}\theN}

\newcommand{\CC}{\mathbb{C}}
\newcommand{\ZZ}{\mathbb{Z}}
\newcommand{\PP}{\mathbb{P}}
\newcommand{\QQ}{\mathbb{Q}}
\newcommand{\FF}{\mathbb{F}}

\newcommand{\EEE}{\mathscr{E}}

\newcommand{\MMM}{{\mathscr{M}}}

\newcommand{\OOO}{{\mathscr{O}}}
\newcommand{\LLL}{{\mathscr{L}}}
\newcommand{\NNN}{{\mathscr{N}}}

\newcommand{\al}{h}
\newcommand{\rk}{\operatorname{rk}}

\newcommand{\SingnG}[1]{\operatorname{Sing}(#1)^{\mathrm{nG}}}
\newcommand{\Cl}{\operatorname{Cl}}
\newcommand{\Clsc}{\operatorname{Cl}^{\mathrm{sc}}}

\newcommand{\Index}{\operatorname{Index}}

\newcommand{\wt}{\operatorname{wt}}
\newcommand{\qq}{\mathbin{\sim_{\scriptscriptstyle{\QQ}}}}

\newcommand{\red}{\operatorname{red}}
\newcommand{\ord}{\operatorname{ord}}

\newcommand{\pr}{{\operatorname{pr}}}

\newcommand{\mult}{{\operatorname{mult}}}

\newcommand{\discr}{{\operatorname{discr}}}

\newcommand{\Supp}{{\operatorname{Supp}}}

\newcommand{\p}{{\operatorname{p}_{\operatorname{a}}}}

\newcommand{\mumu}{{\boldsymbol{\mu}}}

\newcommand{\type}[1]{$\mathrm{#1}$}
\newcommand{\typeDu}[2]{$\mathrm{#1}_{#2}$}
\newcommand{\typem}[1]{$\mathbf{#1}$}
\newcommand{\typeb}[1]{$\mathrm{(#1)}$}
\newcommand{\typek}[2]{$(\mathrm{#1}_{#2}^{\vee})$}
\newcommand{\typeci}[2]{$(\mathrm{#1}_{#2})$}
\newcommand{\typee}[2]{$\mathrm{#1}/{#2}$}

\newcommand{\NE}{\overline{\operatorname{NE}}}

\usepackage[OT2,T1]{fontenc}
\DeclareSymbolFont{cyrletters}{OT2}{wncyr}{m}{n}
\DeclareMathSymbol{\Sha}{\mathalpha}{cyrletters}{"58}
\newcommand{\dif}{\operatorname{d}_{\Sha}}

\newcommand{\comp}\circ
\newcommand{\xref}[1]{\textup{\ref{#1}}}

\swapnumbers
\theoremstyle{plain}
\newtheorem{theorem}[subsection]{Theorem}
\newtheorem{lemma}[subsection]{Lemma}
\newtheorem{proposition}[subsection]{Proposition}
\newtheorem{stheorem}[equation]{Theorem}
\newtheorem{corollary}[subsection]{Corollary}
\newtheorem{scorollary}[equation]{Corollary}

\newtheorem{sclaim}[equation]{Claim}
\newtheorem{slemma}[equation]{Lemma}
\newtheorem{sproposition}[equation]{Proposition}

\theoremstyle{definition}
\newtheorem{definition}[subsection]{Definition}
\newtheorem{setup}[subsection]{Setup}
\newtheorem{sdefinition}[equation]{Definition}
\newtheorem{notation}[subsection]{Notation}
\newtheorem{snotation}[equation]{Notation}
\newtheorem{example}[subsection]{Example}

\newtheorem{subexample}[equation]{Example}
\newtheorem{subexamples}[equation]{Examples}

\newtheorem{sremarks}[equation]{Remarks}
\newtheorem{sremark}[equation]{Remark}

\begin{document}

\title{Birational transformations of threefold $\QQ$-conic bundles
}
\author{Yuri Prokhorov}
\thanks{
The author  was partially supported by the HSE University Basic Research Program. }
\address{
Steklov Mathematical Institute of Russian Academy of Sciences, Moscow, Russia
\newline\indent
Department
of Algebra, Moscow State
University, Russia
\newline\indent
National Research University Higher School of Economics, Moscow, Russia
}
\email{prokhoro@mi.ras.ru}
\begin{abstract}
\textit{A $\QQ$-conic bundle} is a contraction $f: X\to Z$ of a three-dimensional  algebraic variety $X$
to a surface~$Z$ such that
the variety~$X$ has only terminal $\QQ$-factorial singularities, the anticanonical divisor $-K_X$ is~$f$-ample, and $\uprho(X/Z)=1$.
We provide an algorithm to transform a  $\QQ$-conic bundle to its standard form.
\end{abstract}
\maketitle

\section{Introduction}
\textit{A $\QQ$-conic bundle} is a contraction $f: X\to Z$ of a three-dimensional  algebraic variety $X$
to a surface~$Z$ such that
the variety~$X$ has only terminal $\QQ$-factorial singularities, the anticanonical divisor $-K_X$ is~$f$-ample, and $\uprho(X/Z)=1$.
Such fibrations are natural outcomes of the Minimal Model Program; they are very useful 
in the study rationality questions since their fibers are rational curves.
However typically the total space of a $\QQ$-conic bundle is singular
and this issue creates difficulty in application of the well developed techniques
of Prym varieties and intermediate Jacobians \cite{Shokurov:Prym}.
On the other hand, it is known that any $\QQ$-conic bundle can be birationally transformed to 
a standard form \cite[Theorem~1.13]{Sarkisov:82e}.
The aim of this paper is to decompose this birational transformation into 
some elementary ones, so-called Sarkisov links. Our decomposition 
follows certain algorithm which was basically outlined in ~\cite{Avilov:cb}.

A \textit{type \typem{I} Sarkisov link}, whose source
is a $\QQ$-conic bundle, is the following commutative diagram
\begin{myequation}
\label{eq:link:0}
\vcenter{
\xymatrix@C=4em@R=1.3em{
\tilde{X}\ar@{-->}[r]^\chi\ar[d]^p& X'\ar[d]^{f'}
\\
X\ar[d]^f & Z'\ar[dl]_{\al}
\\
Z
}}
\end{myequation}
where $p$, $\chi$, $f'$, and $\al$ satisfy the following conditions:
\begin{enumerate}
\item 
$p$ is an extremal blowup with center on a fiber  \textup(see Definition~\xref{def:ebl}\textup),
\item 
$\chi$ is a birational map that is an isomorphism in codimension $1$,

\item 
$f'$ is a $\QQ$-conic bundle,
and 
\item 
$\al$ is a birational morphism whose exceptional divisor is irreducible.
\end{enumerate}
The \textit{center of the link} is the center of the blowup $p$, i.e. the image $p(E)$ of its exceptional divisor $E \subset \tilde X$.
We say that the link \eqref{eq:link:0} is of \textit{md-type} (minimal discrepancy type) if the center of the link is a point 
of index $r>1$ and the exceptional divisor $E$ has minimal discrepancy $1/r$ over~$X$. In this situation we also say that \eqref{eq:link:0}
is an \textit{md-link}.

\begin{example}[{\cite[Lemma~8]{Avilov:cb}}]
\label{ex:ODP}
The variety~$X$ is given in $\PP^2_{x_1,x_2, x_3}\times
\CC^2_{u,v}$ by the equation
\[
x_1^2+x_2^2+uvx_3^2=0,
\]
and $f: X\to Z$ is the
projection to $Z:=\CC^2$. Such $\QQ$-conic bundles are said to be of type \typeci{IF_1}{} (see~\eqref{eq:cb:uv}).
The singular locus of~$X$ consists of one ordinary double point $P$, the discriminant divisor $\Delta_f$ is given by 
$uv=0$, and the singular fiber $C:=f^{-1}(0)_{\mathrm{red}}$
is a pair of lines.
Let $p : \tilde{X} \to X$ be the blowup of $P$. Then the variety $\tilde{X}$ is non-singular and the divisor
$-K_{\tilde{X}}$ is nef and big over~$Z$.
There exists a flop $\tilde{X} \dashrightarrow X'$ with center the proper transform of $C$ and there exists a $K$-negative extremal Mori contraction $f' : X' \to Z'$ over~$Z$.  Thus, we obtain a type 
\typem{I} Sarkisov link~\eqref{eq:link:0}, where 
$f'$ is a standard conic bundle and
$\al : Z' \to Z$ is the blowup of~$o$. Moreover,  the discriminant divisor $\Delta_{f'} \subset Z'$ of $f'$ is the proper transform of $\Delta_f$.
\end{example}

\begin{example}[{\cite[Example~11.4]{P:rat-cb:e}}]
\label{example-simplest-link}
The variety~$X$ is given in $\PP(1,1,1,2)_{x_1,x_2, x_3, x_4}\times
\CC^2_{u,v}$ by the following two equations (see \cite[\S~12]{MP:cb1} and Theorem~\ref{th-index=2})
\begin{equation*}
\left\{
\begin{array}{lllll}
x_1^2+x_2^2&+&ux_4 &+&q_1(x_1,x_2,x_3; u,v)=0
\\[7pt]
x_1x_2+x_3^2&+&vx_4 &+&q_2(x_1,x_2,x_3; u,v)=0
\end{array}
\right.
\end{equation*}
where $q_1$ and $q_2$ are general quadratic forms in $x_1,x_2,x_3$ with coefficients in the maximal ideal $(u,\, v)\subset \CC\{u,v\}$,
and $f: X\to Z$ is the
projection to
$Z:=\CC^2$ \cite[\S~12]{MP:cb1}. Then~$X$ has a unique singular point
$P=(0,0,0,1; 0,0)$,
which is terminal quotient of type $\frac12 (1,1,1)$.
The central curve has four components $C_1,\dots, C_4$ so that all of them pass through~$P$ and they do not meet each other
elsewhere.
The discriminant curve $\Delta_f$ is given by the vanishing of the determinant of the quadratic form $vq_1-uq_2$ (see Proposition~\ref{prop:comp-discrim}),
it has three smooth analytic branches passing through $0$ and intersecting each other transversely.
Let $p: \tilde{X}\to X$ be the
blowup of~$P$, let $E$ be the exceptional divisor, and let
$\tilde C_i\subset \tilde{X}$ be the proper transform of $C_i$. Then $\tilde{X}$ is smooth,
$E\simeq \PP^2$, and $\OOO_E(E)\simeq \OOO_{\PP^2}(-2)$.. 
Furthermore, the curves $\tilde C_1$, \dots $\tilde C_4$
are disjoint and $K_{\tilde{X}}\cdot \tilde C_i=0$. There exists a flop $\chi: \tilde{X} \dashrightarrow X'$ with center $\tilde C$,
which is the simplest Atiyah-Kulikov flop along each curve $\tilde C_i$, and there exists
a Mori extremal contraction $f': X'\to Z'$ over~$Z$.
We obtain an md-link~\eqref{eq:link:0}, where $f'$ is a standard conic bundle and
$\al: Z'\to Z$ is the blowup of~$o$.
Moreover,  the discriminant divisor $\Delta_{f'} \subset Z'$  is the proper transform of $\Delta_f$.
\end{example}

Our main result is the following theorem.

\begin{theorem} 
\label{thm:link:exist}
Let $f: X\to Z$ be a $\QQ$-conic bundle.
Fix a point $o\in Z$.
\begin{enumerate}
\item
\label{thm:link:exist1}
Then there exists a type~\typem{I} link
\eqref {eq:link:0} whose center is contained in the fiber $f^{-1}(o)$.
The morphism
$\al$ is crepant if the point $o\in Z$ is singular
or a weighted blowup \textup(in suitable coordinates\textup) with weights $(1,a)$, $a\ge 1$ if the point $o\in Z$ is smooth.

\item
\label{thm:link:exist2}
If furthermore~$X$ is not Gorenstein near the fiber $f^{-1}(o)$, then the blowup $p$ can be taken so that the 
link \eqref {eq:link:0} is of md-type. In this case
$\chi$ admits a decomposition
\[
\chi=\chi_{n-1}\comp\cdots\comp \chi_1, \qquad n\ge 2,
\]
where $\chi_1$ is either a flip or a flop and, in the case $n>2$, the maps $\chi_2,\dots, \chi_{n-1}$ are flips.
\end{enumerate}
\end{theorem}

\begin{corollary}
\label{cor:crepant}
Let $f: X\to Z$ be a $\QQ$-conic bundle, where the base~$Z$ is singular.
Then there exists a sequence of md-links 
\begin{equation*}
X/Z=X_0/Z_0 \dashleftarrow X_1/Z_1 \dashleftarrow \cdots \dashleftarrow X_m/Z_m,
\end{equation*}
where 
each $X_i/Z_i$ is a $\QQ$-conic bundle and 
the composition
\[
Z_0 \longleftarrow Z_1\longleftarrow \cdots \longleftarrow Z_m
\]
is the minimal resolution.
\end{corollary}

For applications it would be very useful to have classification of md-links or
at least substantive theorems
describing their structure (see e.g. \cite{P:P11223}). It turns out that the types of links appeared in Corollary~\ref{cor:crepant} are very restrictive.
Except for two cases \typeb{II^\vee} and \typeb{II^\vee{+}II^\vee} they are described in detail in Sections~\ref{Sect:T}--\ref{Sect:ID}.

\begin{corollary}
\label{cor:Gor}
Let $f: X\to Z$ be a $\QQ$-conic bundle, where~$X$ is not Gorenstein. 
Then there exists a sequence of md-links
\begin{equation*}
X/Z=X_0/Z_0 \dashleftarrow X_1/Z_1 \dashleftarrow\cdots \dashleftarrow X_k/Z_k,
\end{equation*}
where each $X_i/Z_i$ is a $\QQ$-conic bundle,
each morphism
$Z_{i+1} \to Z_{i}$ is either crepant
or a weighted blowup with weights $(1,a_i)$,
$X_k$ is Gorenstein, and $Z_k$ is smooth.
\end{corollary}

\begin{corollary}
\label{cor:m-standard}
Let $f: X\to Z$ be a $\QQ$-conic bundle. 
Then there exists a sequence of Sarkisov links of type~\typem{I}
\begin{equation*}
X/Z=X_0/Z_0 \dashleftarrow X_1/Z_1 \dashleftarrow\cdots \dashleftarrow X_n/Z_n,
\end{equation*}
where each $X_i/Z_i$ is a $\QQ$-conic bundle, each morphism
$Z_{i+1} \to Z_{i}$ is either crepant
or a weighted blowup with weights $(1,a_i)$, and
$X_n/Z_n$ is a standard conic bundle.
\end{corollary}
Note that in Corollary~\ref{cor:m-standard} we do not assert that all the links are of md-type.

\subsection*{Acknowledgments.}
The author is very grateful to Alexander Kuznetsov for his interest in the current computation
and useful discussions.

\section{Preliminaries}

We work over the complex number field $\CC$. We employ the following standard notation.
\subsection{Notation}
\begin{itemize}
\item
$a(E,V)$ is the discrepancy of a divisor $E$ over $V$,
\item
$\PP(w_1,\dots,w_n)$ is the weighted projective space,
\item
$\mumu_r$ is the cyclic group of order $r$,
\item
if $\mumu_r$ acts on $\CC^n$ diagonally via $(x_1,\dots,x_n) \longmapsto (\zeta^{a_1} x_1,\dots,\zeta^{a_n} x_n)$, where $\zeta$ is a primitive root of unity, then we say that   $(a_1,\dots,a_n)$ are 
weights of the action and the corresponding quotient is denoted by $\CC^n/\mumu_r(a_1,\dots,a_n)$.
\end{itemize}

Below we collect some useful but non-standard facts about terminal singularities, divisorial contractions, and flips.

\subsection{Terminal singularities}
For the classification of terminal threefold singularities we refer to 
\cite{Mori:term-sing} or \cite{Reid:YPG}. Here we just mention that
Gorenstein terminal singularities are divided into five classes \typeDu{cA}{n}, \typeDu{cD}{n}, \typeDu{cE}{6}, \typeDu{cE}{7}, \typeDu{cE}{8}
and non-Gorenstein ones are divided into six classes \typee{cA}{n}, \typee{cAx}2, \typee{cAx}{4}, \typee{cD}2, \typee{cD}2, \typee{cE}2.

Recall that \emph{Shokurov's difficulty} $\dif(V )$ of a threefold $V$ with terminal
singularities is defined as the number of exceptional divisors on $V$ with
discrepancy $<1$:
\[
\dif(V ):=\# \{E \mid \text{$E$ is exceptional over $V$ with $a(E,V)<1$}\}
\]
(see~\cite[Definition 2.15]{Shokurov:non-van}). It is known that this number is
well-defined and finite.

\begin{subexamples}
\label{ex:diff}
\begin{enumerate}
\item
$\dif(V )=0$ if and only if $V$ is Gorenstein \cite{Kawamata-1992-e-app}.
\item
\label{ex:diff1}
Let $P\in V$ be a threefold terminal point of index~$r$.
Then for any integer $1 \le l \le r- 1$ there is an exceptional divisor with center~$P$ and discrepancy $l/r$ \cite[Corollary~4.8]{Shokurov1993}.
Therefore, $\dif(V )\ge r-1$. 

\item
\label{ex:diff2}
Let $P\in V$ be a threefold terminal cyclic quotient singularity of index~$r$.
Then for any integer $1 \le l \le r$ there is \textit{exactly one} exceptional divisor with center~$P$ and discrepancy $l/r$.
Therefore, $\dif(V )=r-1$. 
\end{enumerate}
\end{subexamples}

\begin{slemma}
\label{lemma:div-cnr:d}
Let $g: V\to V'$ be a birational contraction of normal $\QQ$-Gorenstein varieties.
Assume that the singularities of $V'$ are terminal.
Then for any exceptional over $V$ divisor $E$ we have 
\[
a(E,V')\ge a(E,V).
\]
The inequality is strict if the center of $E$ on $V'$ is a singular point.
\end{slemma}

\begin{proof}
Let $h: W \to V$ be a resolution. Write
\[
K_W=h^*K_V+ \sum a_i E_i,\qquad K_V=g^* K_{V'}+\sum b_i F_i, 
\]
where $a_i=a(E_i,V)$ and $b_i=a(F_i, V')>0$. Then
\[
K_W=h^*g^*K_{V'}+ \sum a_i E_i+\sum b_i g^*F_i,
\]
where the divisor $\sum b_i g^*F_i$ is effective. Hence $a(E_i,V')\ge a_i$.
\end{proof}

\subsection{Divisorial contractions}

\begin{sdefinition}
\label{def:ebl}
Let~$X$ be a threefold with terminal $\QQ$-factorial singularities.
An \emph{extremal blowup} of~$X$ is a Mori extremal divisorial contraction $p: \tilde{X} \to X$,
that is, a birational morphism 
such that 
$\tilde{X}$ has only terminal $\QQ$-factorial singularities and $\uprho(\tilde{X}/X)=1$.
\end{sdefinition}

Note that in the above assumptions the anticanonical divisor $-K_{\tilde{X}}$ must be $p$-ample and the exceptional divisor $E\subset \tilde X$ is irreducible.
The \textit{discrepancy of an extremal blowup} is the discrepancy of~$E$.

\begin{sdefinition}
Let $(P\in X)$ be a threefold terminal singularity of index $r>1$. A \textit{Kawamata blowup} of $P\in X$ is a projective 
birational morphism $p: \tilde{X} \to X$ 
such that $\tilde X$ has only terminal singularities, the exceptional divisor $E\subset \tilde X$ is irreducible, $p(E)=P$, and the
discrepancy of $E$ equals $1/r$. 
\end{sdefinition}

Note that if in this situation the variety $X$ is $\QQ$-factorial, then so $\tilde X$ is and then 
$p$ is an extremal blowup. 

\begin{stheorem}[{\cite{Kawamata-1992-e-app}}, {\cite{Hayakawa:blowup1,Hayakawa:blowup2}}]
\label{thm:w-blowup}
Let $(P\in X)$ be a threefold terminal singularity of index $r>1$. 
Then there exists a local coordinate system and a weight such that the corresponding
weighted blowup $p:\tilde X\to X$ that is Kawamata blowup of~$P$.
If moreover, the singularities of~$X$ are $\QQ$-factorial, then $p$ is an extremal blowup.
Conversely, any Kawamata blowup of~$P$ is a weighted blowup in a suitable local coordinate system and 
for suitable weight.
\end{stheorem}

For the definition and properties of weighted blowups we refer to
\cite{Hayakawa:blowup1}, \cite{Markushevich96:discr} or \cite{Chen-J:Resol}.

\begin{stheorem}[{\cite{Kawamata:Div-contr}}]
\label{thm:Kblup}
Let $g: V\to V'$ be an extremal blowup and $\gcd(r,a) = 1$ and let $E$ be its exceptional divisor.
Assume that $g(E)$ is a cyclic quotient singularity $\frac 1r(1, a,-a)$ with $0 < a < r$ and $\gcd(r,a) = 1$.
Then $g$ coincides with the weighted blowup with weights $\frac 1r(1,a,r-a)$. 
In particular, $g$ is a Kawamata blowup, i.e.
\[\textstyle
a(E,V')=\frac 1 r.
\]
\end{stheorem}

\begin{scorollary}
\label{cor:Kblup:di}
Let $g: V\to V'$ be a divisorial Mori contraction. Then
\[
\dif(V)\ge \dif(V')-1.
\]
The inequality becomes an equality if $g$ is the Kawamata blowup of a cyclic quotient singularity.
\end{scorollary}

\begin{proof}
This is a consequence of   Lemma~\ref{lemma:div-cnr:d} and Theorem~\ref{thm:Kblup}.
\end{proof}

The following easy fact will be used frequently below.
\begin{slemma}
\label{lemma:computation-K.C-2}
Let $p:\tilde X\to X$ be a Kawamata blowup of $P\in X$ and let $E$ be the exceptional divisor.
Let $C\subset X$ be a complete curve and let $\tilde C\subset \tilde X$
be its proper transform. Then
\begin{equation}
\label{equation-computation-K.C-2}
K_{\tilde{X}}\cdot \tilde C=K_X\cdot C+\textstyle{\frac1{r}}E\cdot \tilde C.
\end{equation}
\end{slemma}

\subsection{Flips}
\begin{stheorem}[{\cite[(2.13.3]{Shokurov:non-van}}]
\label{thm:flips}
If $\chi: V \dashrightarrow V^{+}$ is a threefold terminal flip, then
for any divisor $E$ over $V$ with center on the flipping locus one has 
\begin{equation}
\label{eg:discr:flip}
a(E,V )>a(E,V^{+}).
\end{equation} 
\end{stheorem}

\begin{scorollary}[{\cite[Corollary 2.16]{Shokurov:non-van}}]
\label{cor:flips}
Let $\chi: V \dashrightarrow V^{+}$ be a threefold terminal flip. Then
\[
\dif(V )>\dif(V^{+} ).
\]
\end{scorollary}

\begin{scorollary}
\label{cor:flips:irr}
Let $\chi: (V\supset L)\dashrightarrow (V^+\supset L^+)$ be a threefold terminal flip.
If~$\dif(V)-\dif(V^+)=1$, then the flipped curve $L^+$ is irreducible and $V^{+}$ has no Gorenstein singular points 
on $L^+$.
\end{scorollary}

\begin{proof}
Assume that $V^{+}$ is singular at some point $P^{+}$. 
Then there are at least two divisors over $V^{+}$ of discrepancy $1$.
These are exceptional divisors of blowup of a component of $L^+$ and 
a divisor over $P^{+}$ \cite{Markushevich96:discr}. Therefore, by Theorem~\ref{thm:flips} 
we have $\dif(V)-\dif(V^+)\ge 2$, a contradiction.
The same arguments work if the flipped curve is reducible.
\end{proof}

\begin{stheorem}[{\cite[Theorem~13.5]{KM:92}}]
\label{thm:flips:comp}
Let $\chi: (V\supset L)\dashrightarrow (V^+\supset L^+)$ be a threefold terminal flip, where
let $L$ \textup(resp. $L^{+}$\textup) is the flipping \textup(resp. flipped\textup) curve. Then
the number of irreducible components of~$L^{+}$ is less than or equal to those of~$L$.
\end{stheorem}

\begin{sproposition}[cf. {\cite[Example~2.3.2]{Kollar1991-flip-flop}}]
\label{prop:flip:smooth}
Let $\chi: (V\supset L)\dashrightarrow (V^+\supset L^+)$ be a threefold terminal flip, where the flipping curve $L$
is irreducible. Assume that $V$ has a 
unique singular point~$P$ on $L$, which is cyclic quotient of index~$r$, and $-K_V\cdot L=(r-1)/r$.
Then $V^+$ is smooth along~$L^+$. 
\end{sproposition}

\begin{proof}
We regard $(V\supset L)$ and $(V^+\supset L^+)$ as analytic germs. Let $D\subset V$ be small disk 
that intersect $L$ transversely at a general point. Thus $D\cdot L=1$ and $(D+K_V)\cdot L=1/r$.
Then we have $\Clsc(V)\simeq \ZZ\cdot (D+K_V)$, where $\Clsc(V)$ is the subgroup of $\Cl(X)$ 
consisting of Weil divisors which are $\QQ$-Cartier, see \cite[Corollary~1.10]{Mori:flip}
or \cite[Corollary~4.7.3]{MP:1pt}. Therefore, $(r-1)D\sim -rK_V$ and $D\sim r(D+K_V)$. Then as in \cite[Proof of Theorem~4.2]{KM:92}
we take a cyclic~$r$-fold cover $\pi: V'\to V$ branched over $D$. 
We have 
\begin{equation}
\label{eq:K-branch}
\textstyle
K_{V'}\sim \pi^*\left(K_V+ \frac{r-1}r D\right)\sim \pi^*\big((r-1)D+rK_V\big)\sim 0.
\end{equation} 
Therefore, $V'$ is Gorenstein. Moreover, $V'$ is smooth because the only singular point of $V$ is a cyclic quotient.
Thus $V=V'/\mumu_r$, where the fixed point locus of $\mumu_r$ consists of the divisor $D':=\pi^{-1}(D)_{\red}$ and 
an isolated point $P':=\pi^{-1}(P)$. 
Hence the preimage $L':=\pi^{-1}(L)_{\red}$ is irreducible, $L'\simeq \PP^1$, and $\pi_L:L'\to L$ 
is a cyclic~$r$-fold cover branched over $D\cap L$ and~$P$. Further, $K_{V'}\cdot L'=0$ by~\eqref{eq:K-branch}. 
Hence $(V',L')$ is a flopping germ. We obtain the following diagram
\[
\xymatrix@C=3em@R=1.3em{
V' \ar@{-->}[r] \ar[d]^{\pi}& V'^+\ar[d]^{\pi^+}
\\
V\ar@{-->}[r] & V^+
} 
\]
where $V' \dashrightarrow V'^+$ is the flop of $L'$ and $\pi^+$ is the quotient morphism by 
the induced action of~$\mumu_r$. By \cite[Theorem~2.4]{Kollar:flops} the variety $V'^+$ is smooth.
Let $D'^+\subset V'^+$ be the proper transform of $D'$. Then $D'^+$ contains the flopped curve $C'^+$ 
and the action of $\mumu_r$ on $D'^+$ is trivial. 
Thus the fixed point locus of $\mumu_r$ on $V'^+$
is of pure codimension $1$ near $C'^+$. In this situation the quotient $V^+=V'^+/\mumu_r$ must be smooth.
\end{proof}

We will need the following elementary fact.
\begin{slemma}[cf. {\cite[Lemma~A.5]{MP:class}}]
\label{lemma:flip}
Let $\psi: X\dashrightarrow X^+$ be a $D$-flip, where $D$ is an arbitrary $\QQ$-Cartier divisor, 
and let $L\subset X$ be a proper curve that meets the flipping locus but 
is not contained in it.
Let $D^+:=\psi_* D$ and $L^+:=\psi_* L$. Then 
\[
D\cdot L>D^{+}\cdot L^+.
\]
\end{slemma}
\begin{proof}
Let 
\[
\xymatrix@R=1em{
&\hat X\ar[dr]^q\ar[dl]_p&
\\
X\ar@{-->}[rr]&&X^+
}
\]
be a common resolution. Let $\hat D\subset \hat X$
and $\hat L\subset \hat X$
be proper transforms of $D$ and $L$, respectively. We can write 
\[
\hat{D}=p^* D-\sum a_i E_i = q^* D^{+}-\sum a_i^+ E_i,
\]
where the $E_i$ are prime divisors. Note that all these divisors are exceptional for both $p$ and $q$. Then
\[
\hat{D}\cdot \hat L=D\cdot L-\sum a_i E_i\cdot \hat L = D^{+}\cdot L^+-\sum a_i^+ E_i\cdot \hat L.
\]
Therefore,
\[
D\cdot L= D^{+}\cdot L^++\sum (a_i-a_i^+) E_i\cdot \hat L.
\]
It remains to note that $a_i^+\le a_i$ by the Negativity Lemma 
because the divisor
\[
q^* D^{+}-p^* D = \sum (a_i^+-a_i) E_i
\]
is nef.
\end{proof}

\begin{sproposition}
\label{prop:flip:2comp}
Let $(V\supset L)\dashrightarrow (V^+\supset L^+)$ be a threefold terminal flip, where the flipping curve $L$
is reducible. Assume that $V$ has a 
unique non-Gorenstein point~$P$ on $L$, which is of index~$3$.
Then $V^+$ is Gorenstein along~$L^+$, the curve~$L$ has exactly two components, and~$P$ is the only singular point of $V$ on $L$.
Furthermore, assume that~$P$ is a cyclic quotient singularity. If $L^+$ is reducible, then 
$V^+$ is smooth along~$L^+$. If $L^+$ is irreducible, then 
$V^+$ has exactly one singular point on~$L^+$.
\end{sproposition}

\begin{proof}
Let $L_1,\dots,L_n$ be irreducible components of $L$. All of them pass through $P$ (see e.g. \cite[Corollary~6.3.4]{MP:1pt}).
Run the analytic MMP along the components of $L$, that is, on each step we make a flip in an irreducible curve
(cf. \cite[Theorem~13.5]{KM:92}):
\[
V \overset {\chi_1} \dashrightarrow V' \overset {\chi_2}\dashrightarrow V'' \dashrightarrow \cdots
\overset{\chi_{m}}\dashrightarrow V^{(m)} \overset{g} \longrightarrow  V^+,
\]
where $-K_{V^{(m)}}$ is nef and  $g$ is the canonical map to the canonical model.
Let $L_1\subset V$ (resp. $L_1'\subset V'$) be the flipping (resp. flipped) curve of $\chi_1$.
We have $\discr(V')>\discr(V)=1/3$, hence $\Index(V')\le 2$.
For any component $L_i'\subset L'$, $i\neq 1$ we have 
$K_{V'}\cdot L_i'< K_{V}\cdot L_i<0$ by Lemma~\ref{lemma:flip}, hence for $i\neq 1$ each $(V'\supset L_i')$ is a flipping extremal curve germ \cite[Corollary~1.5]{Mori:flip} and $V'$ is not 
Gorenstein \cite[Theorem~6.2]{Mori:flip}. 
Thus $\Index(V')= 2$. Moreover, the germ $(V\supset L_1)$ is primitive \cite[Corollary~A.3.4]{MP:class}, hence so $(V'\supset L_1')$ is. 
This implies that $V'$ has a unique point, say $P'$, of index $2$.
All the components $L_2',\dots,L_n'$ of $L'$ pass through $P'$.
Let $L_2'\subset V'$ (resp. $L_2''\subset V''$) be the flipping (resp. flipped) curve of $\chi_2$.
As above, $\discr(V'')>\discr(V')=1/2$, hence $V''$ is Gorenstein and $m=2$.
Moreover, $V''$ contains no flipping curves, hence $n=2$. It follows from \cite[Theorem~4.2]{KM:92} that
$P'$ is the only singular point of $V'$ lying on $\cup_{i=2}^n L_i'$.
Hence $X$ is smooth outside $L_1$. By symmetry, we obtain that $X$ is smooth outside $\cap_{i=1}^n L_1=\{P\}$.
For the last assertion, we note that  by Corollary~\ref{cor:flips:irr} the varieties $V'$ and $V''$ 
have no Gorenstein singularities if $P\in X$ is a cyclic quotient.
Then $V^+$ is smooth if and only if $g$ is not an isomorphism.
\end{proof}

\section{$\QQ$-conic bundles}
\begin{definition}
\textit{A $\QQ$-conic bundle} is a contraction $f: X\to Z$ of an algebraic threefold to a surface such that
all the fibers are one-dimensional,
the variety~$X$ has only terminal $\QQ$-factorial singularities, the anticanonical divisor $-K_X$ is~$f$-ample, and $\uprho(X/Z)=1$.
In this situation the surface~$Z$ is called the \textit{base} and the threefold~$X$ is called the \textit{total space} of the $\QQ$-conic bundle $f: X\to Z$.
The \emph{discriminant} of a $\QQ$-conic bundle $f: X\to Z$ is the 
curve $\Delta_f\subset Z$ that is the union of the one-dimensional
components of the set 
\[
\{z\in Z \mid \text{~$f$ is not smooth over $z$}\}.
\]
A \textit{standard conic bundle} is a $\QQ$-conic bundle $f: X\to Z$ whose total space is smooth.
In this case the base~$Z$ is also smooth, the discriminant is a normal crossing curve, and the morphism~$f$ is flat
\cite{Sarkisov:82e}.
\end{definition}

In this paper we do not assume that the base of a $\QQ$-conic bundle is a complete surface.
Moreover, typically wee regard $X$ as a neighborhood of a special singular fiber.

\begin{definition}[{\cite{MP:1pt}}]
\label{def:e-germ}
An \emph{extremal curve germ} is the analytic germ $(X\supset C)$ of a threefold~$X$ with terminal singularities along a reduced complete curve $C$
such that there is a 
contraction $f : X\to Z$ with $C = f^{-1}(o)_{\red}$, $o\in Z$ and $-K_X$ is~$f$-ample. 
Furthermore, $(X\supset C)$ is said to be \textit{flipping} if its exceptional locus coincides with $C$, and \textit{divisorial} if its 
exceptional locus is two-dimensional. If~$f$ is not birational, then~$Z$ is a surface and $(X\supset C)$ is called a \textit{$\QQ$-conic bundle germ}.
\end{definition}

\begin{sremarks}
\begin{enumerate}
\item
If $f: X\to Z$ is a $\QQ$-conic bundle, then 
$(X\supset f^{-1}(o)_{\red})$ is a $\QQ$-conic bundle germ for any point $o\in Z$.
\item
Note that in Definition~\ref{def:e-germ} we do not claim that~$X$ is $\QQ$-factorial neither 
that $\uprho(X/Z)=1$. This is because both these properties are not stable under passing from algebraic to analytic category.
\item
Below, considering extremal curve germs we often will replace the germ $(X\supset C)$ with a suitable representative.
\end{enumerate}
\end{sremarks}

In the case where $C$ is irreducible, there is a classification of extremal curve germs, see \cite{KM:92} and \cite{MP:cb1}. Non-Gorenstein ones are divided into several classes: \typeci{k1A}{}, 
\typeci{k2A}{}, 
\typeci{cD/2}{}, \typeci{cAx/2}{}, \typeci{cE/2}{}, \typeci{cD/3}{}, \typeci{IIA}{}, \typeci{II^\vee}{}, \typeci{IE^\vee}{}, \typeci{ID^\vee}{}, \typeci{IC}{}, \typeci{IIB}{}, \typeci{kAD}{}, and 
\typeci{k3A}{}.
In the Gorenstein case extremal curve germs very easy to describe:

\begin{sproposition}
\label{prop:Gor}
Let $(X \supset C)$ be an extremal curve germ such that $X$ is Gorenstein and let $f: (X \supset C)\to (Z\ni o)$ be the corresponding contraction. 
\begin{enumerate}
\item 
\label{prop:Gor1}
Then $(X \supset C)$ is not flipping \cite[Theorem~6.2]{Mori:flip}.
\item 
\label{prop:Gor2}
If $(X \supset C)$ is divisorial, then~$Z$ is smooth and~$f$ is the blowup of a planar c.i. curve $\Gamma\subset Z$ \cite{Cutkosky:contr}, \cite[\S~4]{KM:92}.
\item 
\label{prop:Gor3}
If $(X \supset C)$ is $\QQ$-conic bundle germ, then~$Z$ is smooth and $X$ admits an embedding $X\subset \PP^2\times Z$
so that the fibers of~$f$ are conics in $\PP^2\times z$, $z\in Z$ \cite{Cutkosky:contr}, \cite[Theorem~10.2]{P:rat-cb:e}.
\end{enumerate}
\end{sproposition}

Here is a simple example of Gorenstein $\QQ$-conic bundle germ.

\begin{subexample}
We say that a $\QQ$-conic bundle germ $(X\supset C)$ if of type \typeci{IF}2, if~$X$ has a unique singular point that is a node and 
$C$ is reducible.
In this case in suitable analytic coordinates in $\PP^2_{x_0,x_1,x_2}\times \CC^2_{u,v}$ the variety~$X$ can be given by the equation
\begin{equation}
\label{eq:cb:uv}
x_1^2+x_2^2+uv x_3^2=0. 
\end{equation} 
It is easy to see that the discriminant of $(X\supset C)$ is a normal crossing curve $\{uv=0\}$.
\end{subexample}

\begin{proposition}
\label{prop:cb:uv}
Let $(X \supset C)$ be a $\QQ$-conic bundle germ and let $f:(X \supset C)\to (Z\ni o)$ be the corresponding contraction.
Assume that $o\in Z$ is a smooth point. 
\begin{enumerate}
\item 
\label{prop:cb:uv1}
If $\Delta_f$ has multiplicity $\le 2$ at~$o$, then~$X$ is Gorenstein near $f^{-1}(o)$. 
\item 
\label{prop:cb:uv2}
If $\Delta_f$ is smooth at~$o$, then so is~$X$ near $f^{-1}(o)$. 
\item 
\label{prop:cb:uv3}
If $\Delta_f$ has normal crossing at~$o$, then either~$f$ is a standard conic bundle or has type \typeci{IF}2.
\end{enumerate}
\end{proposition}

\begin{proof}
The question is purely local, so we may assume that~$Z$ is a small neighborhood of~$o$.
Let $Z^o:=Z\setminus \{o\}$ and let $X^o:=f^{-1}(Z^o)$. 
Then we may assume that $X^o$ is smooth and the restriction $f^o: X^o\to Z^o$ is a standard conic bundle.
Hence the sheaf $f_*\OOO_X(-K_{X^o})$ is locally free on $Z^o$.
Let $\EEE: = f_*\OOO_X(-K_X )^{\vee\vee}$.
Then $\EEE$ is a reflexive sheaf of rank~$3$ that is locally free on~$Z^o$.
Since $o\in Z$ is a smooth point, $\EEE$ is locally free at~$o$ as well.
Since $-K_{X^o}$ is relatively very ample on $X^o$ over~$Z^o$, there exists an 
embedding $\iota: X^o \hookrightarrow \PP_Z(\EEE)$ such that $\iota\comp \pr=f$, where $\pr: \PP_Z(\EEE) \to Z$ is the natural projection.
Now, let $\bar X\subset \PP_Z(\EEE)$ be the closure of~$X^o$.
Since~$Z$ is a small neighborhood of~$o$, we may assume that $\EEE$ is trivial 
and $\PP_Z(\EEE)=\PP^2\times Z$. Then $\bar X$ is given by the following equation
\begin{equation}
\label{eq:cb-eq:}
\sum_{0\le i,j\le 2}^{} a_{i,j}(u,v) x_ix_j=0, 
\end{equation}
where $x_0,x_1,x_2$ are homogeneous coordinates in $\PP^2$, $u,v$ are local coordinates on~$Z$,
and $a_{i,j}(u,v)$ are holomorphic functions with $a_{i,j}=a_{j,i}$.
In these settings, the equation of $\Delta_f$ is 
\[
\det \|a_{i,j}(u,v)\|=0.
\]
Now assume that $o\in \Delta_f$ is a double point.
The cases where $o\in \Delta_f$ is smooth or $o\notin \Delta_f$ are easier and left to the reader.
Thus 
\begin{equation*}
\mult_{(0,0)}\big(\det \|a_{i,j}(u,v)\|\big)=2,
\end{equation*} 
hence $a_{i,j}(0,0)\neq 0$
for some $i,j$. This means that the fiber over~$o$ is one-dimensional, i.e. $\bar{f}: \bar X\to Z$ is flat.
Moreover, $\bar X$ is normal and Gorenstein.
The varieties $\bar X$ and~$X$ are isomorphic in codimension $1$. Since the divisors $-K_{\bar X}$ and $-K_X$ are ample,
$\bar X$ and~$X$ are isomorphic (over~$Z$). 
This proves \ref{prop:cb:uv1}.

To prove \ref{prop:cb:uv2} and \ref{prop:cb:uv3}
assume that $X=\bar X$ is singular.
Moreover, we may assume that~$X$ is singular at the point $(0,0,1;0,0)$. 
Using $\OOO_{Z,o}$-linear coordinate change in $\PP^2$ 
we can reduce the equation~\eqref{eq:cb-eq:} to the following
\[
x_0^2+a_{1,1}(u,v) x_1^2+2a_{1,2}(u,v) x_1x_2+a_{2,2}(u,v) x_2^2=0,
\]
where
$\mult_{(0,0)}(a_{2,2})\ge 2$ and $a_{1,2}(0,0)=0$ by our assumption. 
In this situation, the discriminant 
\[
\Delta_f=\{a_{1,1}a_{2,2}-a_{1,2}^2=0\}
\] 
must be singular at the origin.
This proves \ref{prop:cb:uv2}. Finally, if $\Delta_f$ has normal crossing at~$o$,
then the rank of the quadratic part of $a_{1,1}a_{2,2}-a_{1,2}^2$ equals 
$2$.
Then $a_{1,1}(0,0)\neq 0$ and the equation can be reduced to \eqref{eq:cb:uv}.
\end{proof}

The following fact can be easily checked locally.

\begin{slemma}
\label{lemma:St-c-b}
Let $f:X\to Z$ be a standard conic bundle and let $\Gamma\subset Z$ be a smooth curve
that is not contained in $\Delta_f$. 
Then the surface $f^{-1}(\Gamma)$ is normal.
\end{slemma}

\subsection{$\QQ$-conic bundles of index $2$}

\begin{stheorem}[{\cite[\S~12]{MP:cb1}}]
\label{th-index=2}
Let $(X\supset C)$ be a $\QQ$-conic bundle germ and let $f: (X\supset C)\to (Z\ni o)$ be the corresponding contraction.
Assume that the variety~$X$ is not Gorenstein and has only singularities of index $\le 2$.
Furthermore, assume that the point $(Z\ni o)$ is smooth. Then there is an embedding
\begin{equation}
\label{eq-diag-last-2}
\xymatrix@R=1.3em{X \ar@{^{(}->}[r] \ar[rd]_{f}& \PP(1^3,2)\times \CC^2
\ar[d]
\\
&\CC^2}
\end{equation}
such that~$X$ is given by two equations
\begin{equation}
\label{eq-eq-index2}
\begin{array}{l}
\theta_1(u,v) x_4 + q_1(x_1, x_2,x_3; u,v) =0,
\\[7pt]
\theta_2 (u,v)x_4 + q_2(x_1, x_2,x_3; u,v) =0,
\end{array}
\end{equation}
where 
$u,\, v$ are local coordinates on $\CC^2$,\ 
$x_1,\dots,x_4$ are coordinates on $\PP(1^3,2)$,
$\theta_1(u,v)$ and $\theta_2(u,v)$ are holomorphic functions such that 
$\theta_1(0,0)=\theta_2(0,0)=0$, 
$q_1(x_1, x_2,x_3; u,v)$ and $q_2(x_1, x_2,x_3; u,v)$ are ternary quadratic forms in $x_1,x_2,x_3$ with coefficients in $\CC\{u,v\}$ such that 
the forms $q_1(x_1, x_2,x_3; 0,0)$ and $q_2(x_1, x_2,x_3; 0,0)$ are not proportional.
In these settings the only non-Gorenstein point is $P:=(0,0,0,1; 0,0)$.
The central curve has at most four components. All of them pass through~$P$ and they do not meet each other
elsewhere.
\end{stheorem}

\begin{sproposition}
\label{prop:comp-discrim}
In the notation of Theorem~\xref{th-index=2} 
the discriminant curve $\Delta_f$
is given by the vanishing of the determinant of the quadratic form $\theta_1 q_2-\theta_2 q_1$: 
\[
\Delta_f=\{\det(\theta_1 q_2-\theta_2 q_1)=0\}.
\]
\end{sproposition}

\begin{proof}
The assertion is an immediate consequence of the following simple lemma.
\end{proof}

\begin{slemma}
Let $C\subset \PP(1^3,2)$ be the curve given by the equations
\[
q_1 (x_1,x_2,x_3) + c_1 x_4 = q_2 (x_1,x_2,x_3) + c_2 x_4 = 0,
\]
where $q_1$, $q_2$ are ternary quadratic forms and $c_1$, $c_2$ are constants.
Then $C$ is singular if and only if $\det(c_1 q_2-c_2 q_1)=0$.
\end{slemma}

\subsection{$\QQ$-conic bundles over singular base}
\label{subsectSB}

\begin{stheorem}[{\cite{MP:cb1}}]
\label{thm:DuVal}
Let $f: X\to Z$ be a $\QQ$-conic bundle. Then any singularity of~$Z$ 
is Du Val of type \typeDu{A}{}. 
\end{stheorem}

The following simple fact is very useful for working on birational transformations of $\QQ$-conic bundles.

\begin{sproposition}[\cite{Morrison-1985}]
\label{proposition-Morrison}
Let $\al: S'\to S$ be a proper birational contraction between
surfaces with Du Val singularities such that $\uprho(S'/S)=1$. Then
$\al$ is either a crepant contraction or a
weighted blowup with weights $(1,a)$ of a smooth point.
\end{sproposition}

$\QQ$-conic bundle germs over singular base have very precise descriptions. Up to analytic isomorphisms 
every such germ is the quotients of another $\QQ$-conic bundle germ 
$(X^{\flat}\supset C^{\flat})$ over smooth base by a cyclic group. 
Moreover $X^{\flat}$ is of index $1$ or $2$ (see Proposition~\ref{prop:Gor}\ref{prop:Gor3} and Theorem~\ref{th-index=2}).
Below we present the classification following \cite{MP:cb1} and \cite{MP:cb2}.

\begin{stheorem}
\label{thm:tab:sing}
Let $(X\supset C)$ be a $\QQ$-conic bundle germ and let $f: (X\supset C)\to (Z\ni o)$ be the corresponding contraction.
Assume that the point $Z\ni o$ is singular. Then 
$(X\supset C)$ is biholomorphic to the quotient 
\[
(X\supset C)=(X^{\flat}\supset C^{\flat})/\mumu_r,
\]
where $(X^{\flat}\supset C^{\flat})$ is a $\QQ$-conic bundle germ over a smooth base $(Z^{\flat}\ni o^{\flat})$ and the actions of $\mumu_r$ on 
$(X^{\flat}\supset C^{\flat})$ is free in codimension two. The contraction $f^{\flat}: (X^{\flat}\supset C^{\flat})\to (Z^{\flat}\ni o^{\flat})$ is equivariant, the induced $\mumu_r$-action on 
$(Z^{\flat}\ni o^{\flat})$ is free outside $o^{\flat}$, and 
\[
(Z\ni o)=(Z^{\flat}\ni o^{\flat})/\mumu_r.
\]
The variety $X^{\flat}$ is realized as a complete intersection in $\PP\times Z^{\flat}$, where $\PP$ is some \textup(weighted\textup) projective space,
so that $f^{\flat}$ is the projection to the second factor and the action of $\mumu_r$ on $X^{\flat}$ is induced by an action on $\PP\times Z^{\flat}$. The equations of $X^{\flat}$ and weights of the 
$\mumu_r$-action are listed in Table~\xref{tab:sing}.
\end{stheorem}

\begin{scorollary}
In the notation of Theorem~\xref{thm:tab:sing} the base $(Z\ni o)$ is of type \typeDu{A}{r-1}. 
\end{scorollary}

\begin{snotation}[in Table~{\xref{tab:sing}}]
We identify $(Z^{\flat}\ni o^{\flat})$ with a disk in $(\CC^2\ni 0)$ with coordinates $u,\, v$
and let $\mathfrak{m}$ be the maximal ideal $(u,\, v)\subset \CC\{u,\, v\}$.
We let $x_1,\dots,x_n$ be homogeneous coordinates in $\PP$. 
The action of $\mumu_r$ on $\PP\times Z^{\flat}\subset \PP\times \CC^2$ is diagonal and described by the collection of weights
$\wt(x_1,\dots,x_n;u,v)$.
The $q_i=q_i(x_1,x_2,x_2;u,v)$ are ternary quadratic forms in $x_1,x_2,x_3$ with coefficients in the maximal ideal 
$\mathfrak{m}=(u,\,v)\subset \CC\{u,\, v\}$. In the last column $\SingnG{X}$ denotes the collection 
of non-Gorenstein points of $X$.
\end{snotation}

\begin{table}[h]
\renewcommand{\arraystretch}{1.3}
\begin{longtable}{|l|m{0.24\textwidth}|m{0.07\textwidth}|m{0.039\textwidth}|m{0.18\textwidth}|m{0.2\textwidth}|}
\hline
Type& equations of $X^{\flat}$ & $\PP$&~$r$ &$\wt(x_1,\dots,x_n;u,v)$ & $\SingnG{X}$
\\\hline
\typeci{T}{r,a}& &$\PP^1$ & any $\ge 2$ &$(0,1; a,-a)$, $\gcd (r,a)=1$&\mbox{$\frac1r(1,a,-a)$ $\&$} $\frac1r(-1,a,-a)$
\\\hline
\typeci{k2A}{r} & $x_1^2+ux_2^2+vx_3^2=0$&$\PP^2$ & odd $\ge 3$ &
$(a, -1,0;1,-1)$ $r=2a+1$ & \mbox{$\frac 1r(a,-1,1)$ $\&$} \mbox{$\frac 1r(a+1,1,-1)$} 
\\\hline
\typeb{IE^\vee}&
$
\setlength{\arraycolsep}{1.2pt}
\begin{cases}
x_1^2+x_2^2=ux_4 +q_1 
\\[-4pt]
2x_1x_2+x_3^2=vx_4 +q_2
\end{cases}
$
\newline
$q_1\ni ux_4$,  $q_2\ni vx_4$
&$\PP(1^3,2)$&$4$&$(3,1,2,1;1,3)$& $\frac18(5,1,3)$
\\\hline
\typeb{IA^\vee{+}IA^\vee}&
$
\setlength{\arraycolsep}{1.2pt}
\begin{cases}
x_1^2+x_3^2=ux_4 +q_1
\\[-4pt]
x_2^2+x_3^2=vx_4 +q_2
\end{cases}
$
\newline
$q_1\ni ux_4$,  $q_2\ni vx_4$
&$\PP(1^3,2)$&$2$&$(1,1,0,1; 1,1)$& $\frac14(1,1,3)$
\\\hline
\typeb{IA^\vee}&
$
\setlength{\arraycolsep}{1.2pt}
\begin{cases}
x_1^2+x_2^2=ux_4 +q_1
\\[-4pt]
x_3^2=vx_4 +q_2
\end{cases}
$
\newline
$q_1\ni ux_4$,  $q_2\ni vx_4$
&$\PP(1^3,2)$&$2$&$(0,1,0,1; 1,1)$ & $\frac14(1,1,3)$
\\\hline
\typeb{II^\vee{+}II^\vee}&
$
\setlength{\arraycolsep}{1.2pt}
\begin{cases}
x_1^2+x_3^2=\theta x_4+q_1
\\[-4pt]
x_2^2+x_3^2=v x_4+q_2
\end{cases}
$
\mbox{$\theta \in \mathfrak{m}$, $\theta\not \ni u$}
&$\PP(1^3,2)$&$2$&$(1,1,0,1; 1,1)$& \typee{cAx}{4}
\\\hline
\typeb{II^\vee}&
$
\setlength{\arraycolsep}{1.2pt}
\begin{cases}
x_1^2+x_2^2=\theta x_4+q_1
\\[-4pt]
x_3^2=v x_4+q_2
\end{cases}
$
\mbox{$\theta \in \mathfrak{m}$, $\theta\not \ni u$}
&$\PP(1^3,2)$&$2$&$(0,1,0,1; 1,1)$
& \typee{cAx}{4}
\\\hline
\typek{ID}{m}&$x_1^2+x_2^2+\theta x_3^2=0$\newline $\theta\in\mathfrak{m}^2$, $m:=\rk(\theta_{(2)})$
&$\PP^2$ & $2$&
$(0,1,0; 1,1)$ & \typee{cA}2 if $m>0$ \newline\typee{cAx}2 if $m=0$ 
\\\hline
\end{longtable}
\caption{$\QQ$-conic bundles over singular base}
\label{tab:sing}
\end{table}

We often omit subscripts in types \typeci{T}{r,a}, \typeci{k2A}{r}, \typek{ID}{m}: for example,
we write \typeci{T}{r} or even \typeci{T}{} instead of \typeci{T}{r,a} if the corresponding numbers are not important.

\begin{scorollary}
\label{cor:eq:difSho}
In the notation of Theorem~\xref{thm:tab:sing} Shokurov's difficulty $\dif(X)$  is given by  the following table:
\begin{center}\renewcommand{\arraystretch}{1.3}
\begin{tabular}{|c!{\vrule width 1pt}c|c|c|c|c|c|c|}
\hline
Type
&
\typeci{T}{r}, \typeci{k2A}{r}
&
\typeb{IA^\vee}, \typeb{IA^\vee{+}IA^\vee}
&
\typeb{IE^\vee}
&
\typek{ID}2
&
\typek{ID}1
&
\typek{ID}0
&
\typeb{II^\vee}, \typeb{II^\vee{+}II^\vee}
\\\hline
$\dif(X)$ &$2(r-1)$ & $3$ & $7$ &$1$ & $2$ & $1$ or $2$ & $\ge 3$
\\\hline
\end{tabular}
\end{center}
\end{scorollary}

\begin{proof}
This is a consequence of   Example~\ref{ex:diff}.
\end{proof}

\begin{scorollary}[see {\cite[Corollaries~11.8.1 and~11.8.2]{P:rat-cb:e}}]
\label{cor:cb:plt}
Let $f:X\to Z$ be a $\QQ$-conic bundle.
Pick a point $o\in Z$ such that~$f$ is not smooth over~$o$.
\begin{enumerate}
\item 
\label{cor:cb:plt1}
The the point~$o$ is not contained in $\Delta_{f}$ if and only if~$f$ is of type \typeci{T}{r} near $f^{-1}(o)$. 

\item 
\label{cor:cb:plt2}
The curve  $\Delta_{f}$ is smooth at~$o$ if and only if~$f$ is a standard conic bundle near $f^{-1}(o)$ and $f^{-1}(o)$ is a reduced conic.

\item 
\label{cor:cb:plt3}
The pair $(Z,\Delta_{f})$ is lc but not plt at~$o$
if and only if either~$f$ is of type~\typeci{k2A}{r} or~\typek{ID}2 near $f^{-1}(o)$,
or~$f$ is a standard conic bundle near $f^{-1}(o)$ and $f^{-1}(o)$ is a double line.
\end{enumerate}
\end{scorollary}

\begin{scorollary}
\label{cor:cb:plt4}
Let $f:X\to Z$ be a $\QQ$-conic bundle and let $\Delta_1\subset \Delta_f$ be 
an irreducible complete curve such that 
the intersection $\Delta_1\cap \Supp(\Delta_f-\Delta_1)$ is either empty or is a single point.
Then $\p(\Delta_1)>0$.
\end{scorollary}

\begin{proof}
Assume the converse. Then $\Delta_1$ is a smooth rational curve and $f$ is a standard conic bundle 
over $\Delta_1 \setminus \Supp(\Delta_f-\Delta_1)$ by Corollary~\ref{cor:cb:plt}.
Let $\hat \Delta_1$  be the curve parameterizing
irreducible components
of degenerate conics over $\Delta_1$. We have a natural double cover $\tau:\hat \Delta_1\to \Delta_1$,
which is unramified outside $\Delta_1\cap \Supp(\Delta_f-\Delta_1)$.
Then $\tau$ splits. This means that the divisor $f^{-1}(\Delta_1)$ is reducible.
However, the latter contradict our condition $\uprho(X/Z)=1$.
\end{proof}

\section{Proof of Theorem~\xref{thm:link:exist}}
\label{section:ex}

In this section we prove Theorem~\xref{thm:link:exist}.
In particular, we construct Sarkisov link \eqref{eq:link:0}.
The general construction of links is natural and well-known (see e.g. \cite{Corti95:Sark}).
For the definition and basic properties of singularities of pairs 
whose boundary is a combination of linear systems we refer to \cite{Corti95:Sark}, \cite[\S~4]{Kollar95:pairs} or \cite{P:G-MMP}

\begin{notation}
\label{not:SL}
Let $f: X\to Z$ be a $\QQ$-conic bundle.
We regard~$Z$ as a small neighborhood of a point $o\in Z$.
In particular, we assume that~$X$ is smooth outside $f^{-1}(o)$.
\end{notation}

\begin{proof}[Proof of Theorem~\xref{thm:link:exist}~\xref{thm:link:exist1}]
\renewcommand{\qedsymbol}{}
Take two general hyperplane sections $l_1$ and $l_2$ of~$Z$ passing through~$o$.
Let $\LLL$ be the pencil generated by $nl_1$ and $nl_2$ for $n\gg 0$ and let $\MMM:=f^*\LLL$.
Clearly, $\MMM$ has no fixed components.
\end{proof}

\begin{sclaim}
The pair $(X,\MMM)$ is not canonical.
\end{sclaim}
\begin{proof}
Assume the converse, that is, the pair $(X,\MMM)$ is canonical.
Then for a general hyperplane section~$H$ of~$X$ 
the restriction $f_H: H\to Z$ is a finite morphism and the pair $(X,H+(1-\varepsilon )\MMM)$ 
is plt by Bertini's theorem. By the inversion of adjunction the pair $(H, (1-\varepsilon )\MMM)$ is klt.
By the ramification formula we have
\[
K_H+ (1-\varepsilon )\MMM -R= f_H^* (K_Z+ (1-\varepsilon )\LLL),
\]
where $R$ is the ramification divisor. Then the pair $(Z, (1-\varepsilon )\LLL)$ is also klt
(see e.g. \cite[Proposition~3.16]{Kollar95:pairs}). This is not possible for $n\gg 0$
because the multiplicity of a general member of $\LLL$ is al least $n$ at $o$. The contradiction proves the claim.
\end{proof}

\begin{proof}[Proof of Theorem~\xref{thm:link:exist}~\xref{thm:link:exist1} \textup(continued\textup)]
Take $0<c<1$ so that $(X,c\MMM)$ is maximally canonical,
i.e. $(X,c\MMM)$ is canonical and $(X,c'\MMM)$ is terminal for any $c'<c$. Then there exists an extremal blowup
$p: \tilde{X}\to X$ which is crepant with respect to $K_X+c\MMM$, i.e.
\[
K_{\tilde X}+c\tilde \MMM\qq K_X+c\MMM,
\]
where $\tilde \MMM:= p^{-1}_*\MMM$ 
(see \cite[Proposition~2.10]{Corti95:Sark} or \cite[Claim~4.5.1]{P:G-MMP}). Let $E\subset \tilde X$ be the $p$-exceptional divisor. 
Run the $(K_{\tilde X}+c\tilde \MMM)$-MMP on $\tilde X$ over~$Z$.
On each step the property of $\tilde X$ to have terminal $\QQ$-factorial singularities is preserved.
After a number of log flips (which are in fact flips, flops, and antiflips) we obtain a non-birational 
Mori contraction $f': X'\to Z'$.

Assume that $f'$ is divisorial. 
Since $\uprho(Z'/Z)=1$, the morphism $\al: Z'\to Z$ is a $\QQ$-conic bundle. 
Then the $f'$-exceptional divisor must coincide with the proper transform of $E$.
Hence, the map $g= f'\comp \chi\comp p^{-1}:X \dashrightarrow Z'$ is an isomorphism in codimension $1$. 
Since divisors $-K_{X}$
and $-K_{Z'}$ are ample over~$Z$, the map $g$ must be an isomorphism. Clearly, $g_*\MMM=f'_*\chi_*\tilde \MMM$.
On the other hand, since $p$ is crepant with respect to $K_X+c\MMM$
and $f'$ is $(K_{X'}+c\chi_*\tilde \MMM)$-negative we have
\[
a(E,X, c\MMM)=0,\qquad a(E, Z', c g_*\MMM)>0
\]
and for any exceptional divisor $F$ over $\tilde X$ we have
\[
0\le a(F,X, c\MMM)=a(F, \tilde X, c\tilde \MMM)\le a(F, X', c\chi_*\tilde \MMM) \le a(F, Z', c g_*\MMM).
\]
Therefore, the number of log crepant divisors of $(Z', cf'_*\chi_*\tilde \MMM)$ 
is strictly less that this number of $(X,c\MMM)$. In other words, the singularities of 
$(Z', cf'_*\chi_*\tilde \MMM)$ 
are ``better'' than those of $(X,c\MMM)$.
On the other hand, $(Z', cf'_*\chi_*\tilde \MMM)\simeq (X,c\MMM)$.
The contradiction shows that $f'$ is not a divisorial contraction.
Then the only possibility is that $f'$ is a $\QQ$-conic bundle.
We obtain the link \eqref{eq:link:0}. The last assertion follows from Proposition~\ref{proposition-Morrison}
because the singularities of~$Z$ and $Z'$ are at worst Du Val (Theorem~\ref{thm:DuVal}).
\end{proof}

Now we prove the assertion \xref{thm:link:exist2} of 
Theorem~\xref{thm:link:exist}. It is consequence of Propositions~\ref{prop:link} and~\ref{prop:Q:w-blowup} below.

For the link \eqref{eq:link:0} we employ 
the following additional notation.

\begin{notation}
\label{not:SL1}
By  $E$ we denote the $p$-exceptional divisor and we let $E'$ be the proper transform of $E$ on $X'$
Denote also $\Gamma:=f'(E')$.
Thus $\Gamma$ is the $\al$-exceptional divisor.
Let
$\tilde C\subset \tilde{X}$ be the proper transform of $C$
\end{notation}

\begin{proposition}
\label{prop:link}
Let $f: X\to Z$ be a $\QQ$-conic bundle. Fix a point $o\in Z$.
Let $p: \tilde{X}\to X$ be an extremal blowup whose center is a point on $f^{-1}(o)$. Assume that $-K_{\tilde X}$ is nef over~$Z$.
Then there exists a type~\typem{I} Sarkisov link \eqref{eq:link:0},
where $f'$ is a $\QQ$-conic bundle and $\al$ is a birational contraction with
$\uprho(Z'/Z)=1$, and $\chi$ admits a decomposition $\chi=\chi_{n-1}\comp\cdots\comp \chi_1$, where $n\ge 2$, 
$\chi_1$ is either a flip or a flop and, in the case $n>2$, the maps $\chi_2,\dots, \chi_{n-1}$ are flips.
\end{proposition}

\begin{proof}
By our assumptions the divisor $-K_{\tilde X}$ is nef and big over~$Z$ for  $0<\epsilon\ll 1$.
In particular, $\tilde{X}$ is a
variety of type FT over~$Z$ \cite{P-Sh:JAG} and the divisor $-(K_{\tilde X}-\epsilon E)$ is ample over~$Z$
for some $0<\epsilon \ll 1$ because $E\cdot \tilde C>0$.
Run the $(K_{\tilde X}-\epsilon E)$-MMP on $\tilde{X}$ over~$Z$ in the direction different from the contraction to~$X$.
More precisely, it works as follows. 
Since $p(E)$ is a point, no fibers of the first contraction 
$\upsilon_1:\tilde X=X_1\to\bar X_1$
are contained in $E$.
Since $p^{-1}(f^{-1}(o))=E\cup \tilde C$, the contraction $\upsilon_1$ must be small.
If $-K_{\tilde X}$ is ample, then the first operation must be a flip whose center is contained in $\tilde C$.
Otherwise the first operation is a flop whose center is again contained in $\tilde C$.
Thus we have a sequence of flops and flips whose centers are contained in the fiber over~$o$.
This sequence terminates with $X':=X_n$ so that 
there exists a non-small Mori extremal contraction $f': X'\to Z'$:
\begin{equation}
\label{eq:flip}
\vcenter{
\xymatrix@C=1.7em@R=1.3em{
&\tilde X=X_1\ar[dl]_{p=\theta_1}\ar[dr]^{\upsilon_1}\ar@{-->}[rr]^{\chi_1}&&X_2\ar[dl]_{\theta_2}\ar[dr]^{\upsilon_2}\ar@{-->}[r]^{\chi_2}&
&\cdots&\ar@{-->}[r]^{\chi_{n-1}}&X_n=X'\ar[dl]_{\theta_n}\ar[dr]^{f'}
\\
X&&\bar X_1&&\bar X_2&\cdots&\bar X_{n-1} && Z'
}}
\end{equation}
Let $\Theta_i$ (resp. $\Upsilon_i$) be an
exceptional curve of $\theta_i$ (resp. $\upsilon _i$) and let $\Lambda_i\subset X_i$ be a general fiber over~$Z$. 
Here for uniformity we put $\theta_1=p$ and $\upsilon_n=f'$. 
Since $\uprho(X_i/Z)=2$, we have the numerical equivalence
\begin{equation}
\label{eq:Lambda:tmp}
\Lambda_i\equiv a_i\Theta_i+b_i\Upsilon_i
\end{equation} 
over~$Z$ with $a_i,\, b_i\ge 0$.
Since the contractions $\theta_i$ are birational, the curve $\Theta_i$ is not nef for $i=1,\dots, n$.
Hence the classes of
$\Theta_i$ and $\Lambda_i$ are not proportional, i.e. $b_i>0$ for $i=1,\dots, n$. The same arguments show that 
$a_i>0$ for $i=1,\dots, n-1$.
Let $E_i\subset X_i$ be the proper transform of $E$. Then we have $E_i\cdot \Lambda_i=0$.
Since $E\cdot \Theta_1<0$, by induction on $i$ one can show that $E_i\cdot \Upsilon_i>0$ for $i=1,\dots, n-1$ 
and $E_i\cdot \Theta_i<0$ for $i=1,\dots, n$. Moreover, $E_n\cdot \Upsilon_n\ge 0$, i.e. the divisor 
$E':=E_n$ is non-negative on the fibers of $f'$. 
Hence, $E'$ cannot be contracted $E'$ if $f'$ is birational.
Since $E'$ is the only two-dimensional component 
of the fiber over~$o$ of the composition $X'\to Z'\to Z$, the contraction $f'$ is not divisorial. 
Then the only possibility is that $f'$ is a $\QQ$-conic bundle. 

It remains to prove the last statement.
Suppose that $\chi_k$ is a flop for some $k>1$. Then $K_{X_k}$ is trivial on $\Upsilon_k$.
Since $K_{X_k}\cdot \Lambda_k<0$, we have $K_{X_k}\cdot \Theta_k<0$ by \eqref{eq:Lambda:tmp}. Hence,
$K_{X_{k-1}}\cdot \Upsilon_{k-1}>0$, i.e. $\chi_{k-1}$ is
an antiflip, a contradiction.
\end{proof}

\begin{scorollary}
In the above notation $\chi_i$ is a $(-E_i)$-flip for $i=1,\dots, n$. 
\end{scorollary}

\begin{scorollary}
In the above notation $E_i$ is the whole fiber over $o\in Z$ for $i=2,\dots, n$
and the divisor $-E'=-E_n$ is nef over~$Z$. 
\end{scorollary}

The arguments in proof of the following proposition is very similar to the ones of \cite[\S~3]{Chen-Hacon:Factoring}.
For convenience of the reader we provide a complete proof adopted to our situation. 

\begin{proposition}
\label{prop:Q:w-blowup}
Let $f: X\to Z$ be a $\QQ$-conic bundle, where~$X$ is not Gorenstein, let $P\in X$ be a point of maximal index, and let
$p: \tilde X\to X$ be a Kawamata blowup of $P$. 
Then the divisor $-K_{\tilde X}$ is nef over~$Z$.
\end{proposition}

\begin{proof}
The statement is local on~$Z$, so we assume that~$Z$ is a small neighborhood of $o:=f(P)\in Z$.
By Theorem~\ref{thm:w-blowup} there exists 
a weighted blowup $p:\tilde X\to X$ of $P\in X$ with irreducible exceptional divisor which is 
a Kawamata blowup.
Let $C:=f^{-1}(o)_{\red}$, let $C_1,\dots,C_m$ be irreducible components of $C$
and let $\tilde C_1,\dots,\tilde C_m\subset \tilde X$ be their proper transforms.
It is sufficient to show that $K_{\tilde X}\cdot \tilde C_i\le 0$ for $i=1,\dots, m$.
Since for each $i$ component $C_i\subset C$ the analytic germ $(X\supset C_i)$ 
is an extremal curve germ \cite[Corollary~1.5]{Mori:flip}. Therefore, the assertion is a consequence of Lemma~\ref{lemma:w-blowup} below. 
\end{proof}

\begin{lemma}[cf. {\cite[Theorem~3.1]{Chen-Hacon:Factoring}}]
\label{lemma:w-blowup}
Let $(X\supset C)$ be an extremal curve germ with \emph{irreducible} central curve $C$ and non-Gorenstein total space~$X$
and let $P\in X$ be a point of maximal index.
Let $p: \tilde X\to X$ be a Kawamata blowup of $P$. Then $K_{\tilde X}\cdot \tilde C\le 0$, where $\tilde C$ denotes the proper transform 
of $C$ in $\tilde X$.
\end{lemma}

\begin{proof}
It is known that a general element $D\in |-K_{(X,C)}|$
of the anti-canonical linear system of a tubular analytic neighborhood $X\supset C$ has only Du Val singularities \cite{KM:92}, \cite{MP:cb1,MP:cb3}. Thus the pair $(X,D)$ is canonical near $C$.
We can write
\[
K_{\tilde X}+\tilde D=p^*(K_X+D)+ a E, 
\]
where $\tilde D$ is the proper transform of $D$, $a\in \frac1r \ZZ$, and $a\ge 0$. On the other hand, $a< a(E,X)=1/r$
because $p$ is a Kawamata blowup.
Hence, $a=0$. If $D$ does not contain $C$, then $\tilde D\not\supset \tilde C$ and so
\[
-K_{\tilde X}\cdot \tilde C = \tilde D\cdot \tilde C\ge 0.
\]
Thus we may assume that $D\supset C$.
By \cite[Theorem~2.2]{KM:92} and \cite{MP:25-corr} (see also \cite{MP:1pt}) 
the only cases to consider are the cases where $(X\supset C)$ is of type
\typeb{IC}, \typeb{IIB}, \typeb{k2A}, \typeb{k3A}, or \typeb{kAD}.
Below we consider these cases separately.

The following fact is very useful. It will be applied frequently below.
\begin{sremark}
\label{rem:EC}
Since $C$ is a smooth curve, the intersection $E\cap \tilde C$
is a single point and $\tilde C$ is smooth at $E\cap \tilde C$.
If furthermore $E$ is smooth at $E\cap \tilde C$, then $\tilde X$ is also smooth at $E\cap \tilde C$ and $E\cdot \tilde C=1$.
\end{sremark}

\subsection*{Case \typeb{IC}}
In this case~$P$ is a unique singular point of~$X$. 
Moreover, $-K_X\cdot C=1/r$ by \cite[Theorem~6.5 and (2.3.2)]{Mori:flip} and \cite[Corollary~4.7.4]{MP:1pt}.
By \cite[Appendix~A]{Mori:flip}, in a neighborhood of~$P$ we have an identification
\[
(C\subset X)= \left(\{y_1^{r-2}-y_2^2=y_4=0\}\subset \CC^3_{y_1,y_2,y_4}\right)/\mumu _r (2,r-2,1),
\]
where $r\ge 5$ is odd.
In particular, $P\in X$ is a cyclic quotient of type $\frac1r (2,r-2,1)$.
Hence the only choice of 
$p: \tilde X\to X$ is the $\frac 1 r (2,r-2,1)$-weighted blowup of~$P$ (see Theorem~\ref{thm:Kblup}).
Then the exceptional divisor $E$ is isomorphic to the 
weighted projective plane $\PP(2,r-2,1)$ and the point $E\cap \tilde C$ is given by $y_1^{r-2}-y_2^2=y_4=0$.
Hence $E\cap \tilde C$ lies in the smooth part of $E$, as well as, in the smooth part of $\tilde X$.
Therefore, $E\cdot \tilde C=1$ and 
$K_{\tilde{X}}\cdot \tilde C=0$ by Remark~\ref{rem:EC} and~\eqref{equation-computation-K.C-2}.

\subsection*{Case \typeb{IIB}}
In this case $P$ is the only singular point of~$X$ (see
\cite[Theorem 6.7]{Mori:flip}, \cite[Theorem 8.6, Lemma 7.1.2]{MP:cb1}).
By \cite[(2.3.2), Corollary 2.10]{Mori:flip} we have $-K_X\cdot C=1/4$.
By \cite[Appendix~A]{Mori:flip}, in a neighborhood of~$P$ we have identifications
\[
(C\subset X)= \left(\{y_1^2-y_2^3=y_3=y_4=0\}\subset \{ y_1^2-y_2^3+\gamma (y_1,y_2,y_3,y_4)=0 \} \right)/\mumu_4(3,2,1,1),
\]
where $\gamma$ is a semi-invariant with $\wt(\gamma)\equiv 2\mod 4$, $\gamma\in (y_3,y_4)$, and $\gamma\notin (y_1,\dots,y_4)^3$.
The latter means that $\gamma_{(2)}(0,0,y_3,y_4)\neq 0$, where $\gamma_{(2)}$ is the degree $2$ homogeneous part of $\gamma$.
In particular,~$P$ is a singularity of type \typee{cAx}{4}.
We may assume that $\gamma_{(2)}$ does not contain the term $y_3^2$.
Let 
\[
\textstyle
w:= \frac 1 4 (3,2,1,5).
\]
The only choice of Kawamata blowup $p: \tilde X\to X$ is the $w$-weighted blowup of~$P$
\cite[\S~7]{Hayakawa:blowup1}.
Then the exceptional divisor is given in $\PP(3,2,1,5)$ by
\[
y_1^2-y_2^3+ y_3\gamma'(y_1,\dots, y_4)=0,
\]
where $\gamma'$ is the quasi-homogeneous part of $\gamma$ of $w$-weighted degree $5$.
and the point $E\cap \tilde C$ is given by $y_1^2-y_2^3=y_3=y_4=0$.
Hence, as above, $E\cap \tilde C$ lies in the smooth part of $E$, as well as, in the smooth part of $\tilde X$.
Therefore, $E\cdot \tilde C=1$ and 
$K_{\tilde{X}}\cdot \tilde C= 0$ by  Remark~\ref{rem:EC} and~\eqref{equation-computation-K.C-2}.

\subsection*{Cases \typeb{kAD} and \typeb{k3A}}
In these cases~$X$ has two non-Gorenstein points~$P$ and $Q$ of indices $r\ge 3$ and $2$, respectively, and 
in the case \typeb{k3A}\ $X$ has also one Gorenstein singular point.
Here~$r$ is odd and in a neighborhood of~$P$ we have identifications
\[
(C\subset X)= \left(\{\text{$y_1$-axis}\}\subset \CC^3_{y_1,y_2,y_3} \right)/\mumu_r(1, (r+1)/2,-1).
\]
(see \cite[Lemma~2.12.2 and Proposition~9.4]{KM:92}, \cite{Mori:errKM}, \cite[Lemmas~5.3 and~6.12]{MP:cb3}). 
In particular, $P\in X$ is a cyclic quotient of type $\frac1r (1, (r+1)/2,-1)$.
Moreover, $-K_X\cdot C=1/{2r}$ \cite{Mori:errKM}, \cite[Corollary~4.7.4]{MP:1pt}.
As above, the only choice for Kawamata blowup $p: \tilde X\to X$ is the $\frac 1 r (2,1,r-2)$-weighted blowup of~$P$ (see Theorem~\ref{thm:Kblup}).
Then the exceptional divisor $E\subset X$ is covered by three affine charts $U_{y_1}$, $U_{y_2}$, $U_{y_3}$
and in the chart containing $E\cap \tilde C$ we have the following identifications
\[
(\tilde C\subset E\subset U_{y_1}) =\left(\{\text{$\tilde y_1$-axis}\}\subset \{\tilde y_1=0\}
\subset \CC^3_{\tilde y_1,\tilde y_2,\tilde y_3}\right)/\mumu_2(1,1,1).
\]
This shows that $E\cdot \tilde C=1/2$ and 
$K_{\tilde{X}}\cdot \tilde C= 0$ by \eqref{equation-computation-K.C-2}.

\subsection*{Case \typeb{k2A}}
Then~$X$ has exactly two singular points~$P$ and $Q$ of indices $r>1$ and $m>1$, respectively.
By our assumption $r\ge m$. 
By \cite[\S~9]{Mori:flip} and \cite[\S~10]{MP:cb1}
in a neighborhoods of~$P$ we have an identification
\[
(C\subset X)= \big(\{\text{$y_1$-axis}\}\subset \{y_1y_2+\phi(y_3^r,y_4)=0\} \big)/\mumu_r(1, -1, a,0),
\]
where $\gcd(r,a)=1$.
Let $k:=\ord(\phi)$ for $\wt(y_3,y_4)=(1/r,1)$.
Take $a'$ so that $aa'\equiv 1\mod r$ and $0<a'<kr$, and $a'':=kr-a'$.
Then any Kawamata blowup of~$P$ is the $w$-weighted blowup of~$P$ with weight
\[\textstyle
w:= \frac1r\big(a', a'', 1, r\big).
\]
(see \cite{Kawamata:Div-contr}, \cite{Kawamata-1992-e-app},\cite[\S~6]{Hayakawa:blowup1}).
In the affine chart $U_{y_1}=\{y_1\neq 0\}$ containing $E\cap \tilde C$ we have identifications
\[
(\tilde C\subset E\subset U_{\tilde y_1}) =\left(\{\text{$\tilde y_1$-axis}\} \subset \{\tilde y_1=0\} \subset \right\{\tilde y_2+\phi(\tilde y_3^ry_1,\, \tilde y_4 \tilde y_1)y_1^{-k/r}=0\left\} 
\right)/\mumu_{a'}(-r, a'', 1, r).
\]
Therefore, 
\begin{equation}
\label{eq::k2A:EC}
E\cdot \tilde C=1/a'. 
\end{equation}
Now again by \cite[\S~9]{Mori:flip} and \cite[\S~10]{MP:cb1}
in a neighborhoods of~$Q$ we have an identification
\[
(C\subset X)=\big(\{\text{$z_1$-axis}\}\subset \{z_1z_2+\psi(z_3^m,z_4)=0\} \big)/\mumu_m(1, -1, b,0),
\]
where $\gcd(m,b)=1$. 
Put 
\[
\delta:= am+br-mr.
\]
By \cite[(2.3.3)]{Mori:flip} and \cite[(3.1.1) and Corollary~4.4.3]{MP:cb1} we have 
\[
-K_X\cdot C=1-w_P(0)-w_Q(0),
\]
where the invariants $w_P(0)$ and $w_Q(0)$ are defined in \cite[\S~2]{Mori:flip}. In our case 
we have $w_P(0)=(r-1)/r$ and $w_Q(0)=(m-1)/m$ by \cite[Corollary~2.10]{Mori:flip}.
Thus,
\[\textstyle
-K_X\cdot C= \frac {\delta}{rm} >0,\qquad \text{hence $\delta>0$}.
\]
Note that
\[ 
aa' \delta \equiv \delta \equiv am \mod r.
\] 
Hence $a' \delta \equiv m\mod r$ and
$a' \delta \ge m$ because $m \leq r$.
Then from \eqref{equation-computation-K.C-2} and \eqref{eq::k2A:EC} we obtain
\[\textstyle
-K_{\tilde X}\cdot \tilde C= \frac {\delta}{rm}- \frac1{ra'}=
\frac{a'\delta - m}{a'rm}\ge 0.\qedhere
\]
\end{proof}

\begin{sremarks}
\label{rem:un}
\begin{enumerate}
\item
\label{rem:un1}
In the cases \typeb{IC}, \typeb{IIB}, \typeb{kAD}, and \typeb{k3A} 
the choice of Kawamata blowup $p$ is unique. In the case \typeb{k2A} there are exactly $k$ 
choices of $p$.
\item
\label{rem:un2}
In the cases \typeb{IC}, \typeb{IIB}, \typeb{kAD}, and \typeb{k3A} the curve $\tilde C$ is  flopping.
\end{enumerate}
\end{sremarks}

\begin{proof}[Proof of Theorem~\xref{thm:link:exist}\xref{thm:link:exist2}]
By Proposition~\xref{prop:Q:w-blowup} there exists an extremal blowup $p:\tilde X\to X$ 
satisfying the required properties. Then the link exists by 
Proposition~\xref{prop:link}.
\end{proof}

Note that in many cases our construction of an md-link in Theorem~\xref{thm:link:exist}\xref{thm:link:exist2}
is canonical:

\begin{sremark}
In the notation of Theorem~\xref{thm:link:exist}\xref{thm:link:exist2} regard $(X\supset C)$ as an analytic neighborhood
and let $D\in |-K_X|$ be  a general member.
\begin{enumerate}
\item 
Assume that $D$ does not contain any component of $C$. If, in this situation, $D\cap C$ is a single point $P$, then
the uniqueness of an md-link depends on  the uniqueness of Kawamata blowup of $P\in X$ (see \cite{MP:1pt} and \cite{Hayakawa:blowup1}).
Otherwise the germ $(X\supset C)$ is of type~\typeci{T}{}  and there are exactly two md-links starting from it 
(see \S~\ref{Sect:T}).
\item 
Assume that $D$  contains a component $C_i\subset C$.
If the germ $(X\supset C_i)$ is of type \typeb{IC}, \typeb{IIB}, \typeb{k2A}, \typeb{k3A}, or \typeb{kAD}, then
$D\cap C=C_i$ and the point $P\in X$ of maximal index is unique \cite{MP:class}. 
In this case, it follows from the proof of Lemma~\ref{lemma:w-blowup} (see Remark~\ref{rem:un}\ref{rem:un1})
that an  md-link is unique. Otherwise $(X\supset C_i)$ is of type \typeb{k2A}
and for any other $C_j\subset C$ the germ $(X\supset C_j)$ is of type \typeb{k2A} or \typeb{k1A} \cite{MP:class}. 
This case is more difficult to analyze but in general   an  md-link is not unique (see \S~\ref{Sect:k2} for the case of singular base).
\end{enumerate}
\end{sremark}

\begin{proof}[Proof of Corollary~\xref{cor:crepant}]
Apply Theorem~\xref{thm:link:exist}~\xref{thm:link:exist2} over singular points of~$Z$.
We get a sequence 
\[
X/Z=X_0/Z_0 \dashleftarrow X_{m}/Z_{m}
\]
of md-links such that morphisms
$Z_{i+1} \to Z_{i}$ are crepant.
Since the number of crepant divisors on~$Z$ is finite, the sequence of
terminates.
\end{proof}

\begin{proof}[Proof of Corollary~\xref{cor:Gor}]
By Corollary~\xref{cor:crepant} we may assume that~$Z$ is smooth.
If $\Delta_f$ has at most double points as singularities, then~$X$ is Gorenstein by Proposition~\ref{prop:cb:uv}\ref{prop:cb:uv1}.
This contradicts our assumption. Thus $\Delta_f$ has a point~$o$ of multiplicity $\ge3$
such that~$X$ is not Gorenstein over~$o$.
Applying Theorem~\xref{thm:link:exist}~\xref{thm:link:exist2} we obtain an md-link 
$X/Z \dashleftarrow X'/Z'$ whose center is a non-Gorenstein point $P$ lying over $o$. 
The morphism $Z'\to Z$ is weighted blowup of $o\in Z$.
If $Z'$ is singular, we apply Corollary~\xref{cor:crepant} and get a sequence of md-links
$X/Z \dashleftarrow X^{(1)}/Z^{(1)}$, where $Z^{(1)}$ is smooth. 
Otherwise we put $X^{(1)}/Z^{(1)}=X'/Z'$. Thus $f^{(1)}: X^{(1)}\to Z^{(1)}$ is a $\QQ$-conic bundle, $Z^{(1)}$ is smooth,
and there is a birational morphism $g: Z^{(1)}\to Z$ of smooth surfaces such that $\Delta_f=g_* \Delta_{f^{(1)}}$.
In particular, $\Delta_{f^{(1)}} \subset g^{-1} (\Delta_f)$.
The morphism $g: Z^{(1)}\to Z$ passes through the usual blowup of~$o$.
Repeating the procedure, we get sequence of 
transformations 
\[
X/Z \dashleftarrow X^{(1)}/Z^{(1)} \dashleftarrow \cdots \dashleftarrow X^{(N)}/Z^{(N)}
\]
where each $f^{(i)}: X^{(i)}\to Z^{(i)}$ is a $\QQ$-conic bundle, $X^{(i)}/Z^{(i)} \dashleftarrow X^{(i+1)}/Z^{(i+1)}$ is a sequence of
md-links, $g_i: Z^{(i+1)} \to Z^{(i)}$ is a birational morphism of smooth surfaces
that passes through the blowup of a point $o_i\in \Delta_{f^{(i)}}$ of multiplicity $\ge 3$.
Moreover, $\Delta_{f^{(i+1)}} \subset g_i^{-1} (\Delta_{f^{(i)}})$.
Hence after a number of steps we get a situation where $\Delta_{f^{(N)}}$ has only double points.
Then by Proposition~\ref{prop:cb:uv}\ref{prop:cb:uv1} the variety $X^{(N)}$ is Gorenstein.
\end{proof}

\begin{proof}[Proof of Corollary~\xref{cor:m-standard}]
By Corollary~\xref{cor:Gor} we may assume that~$Z$ is smooth and~$X$ is Gorenstein.
If $\Delta_f$ not a normal crossing curve at a point $o\in \Delta_f$, 
then we apply
Theorem~\xref{thm:link:exist}~\xref{thm:link:exist1} and get a type \typem{I} link 
$X/Z \dashleftarrow X'/Z'$ whose center is lies on $f^{-1}(o)$. 
The morphism $Z'\to Z$ is weighted blowup of $o\in Z$ with weights $(1,a)$, $a\ge 1$.
If $Z'$ is singular (i.e. $a>1$), we apply Corollary~\xref{cor:crepant} and get a sequence of md-links
$X/Z \dashleftarrow X^{(1)}/Z^{(1)}$, where $Z^{(1)}$ is smooth. 
Otherwise we put $X^{(1)}/Z^{(1)}=X'/Z'$. As in the proof of Corollary~\xref{cor:Gor}, we obtain sequence of 
transformations 
\[
X/Z \dashleftarrow X^{(1)}/Z^{(1)} \dashleftarrow \cdots \dashleftarrow X^{(N)}/Z^{(N)}
\]
where each $f^{(i)}: X^{(i)}\to Z^{(i)}$ is a $\QQ$-conic bundle, $X^{(i)}/Z^{(i)} \dashleftarrow X^{(i+1)}/Z^{(i+1)}$ is a sequence of
type \typem{I} links, $g_i: Z^{(i+1)} \to Z^{(i)}$ is a birational morphism of smooth surfaces
that passes through the blowup of a non-normal crossing point $o_i\in \Delta_{f^{(i)}}$.
Moreover, $\Delta_{f^{(i+1)}} \subset g_i^{-1} (\Delta_{f^{(i)}})$.
Hence after a number of steps we get a situation where $\Delta_{f^{(N)}}$ is a normal crossing curve.
Then by Proposition~\ref{prop:cb:uv}\ref{prop:cb:uv3} 
over any point $o_N\in Z^{(N)}$ our $\QQ$-Conic bundle is either standard of have type \typeci{IF}2.
In the latter case we apply Example~\ref{ex:ODP}
to get a standard model.
\end{proof}

\section{$\QQ$-conic bundles of type \typeci{T}{}}
\label{Sect:T}
In this section we employ the following notation. 
\begin{setup}
\label{setup:Sect:T}
Let $f: X\to Z$ be a $\QQ$-conic bundle, which is of type \typeci{T}{r,a}, $0<a<r$ over a point $o\in Z$
and is a standard conic bundle outside~$o$. Let $C:=f^{-1}(o)_{\mathrm{red}}$ be the central fiber.
Recall that in this situation
the tubular neighborhood $(X\supset C)$ is biholomorphic to the quotient of $\PP^1_{x_1, x_2}\times
\CC^2_{u,v}$ by the $\mumu_r$-action with weights
\[
\wt(x_1, x_2;\, u,v)=(0,\,1;\,a,-a),
\]
where $\gcd (r,a)=1$.
The singular locus of~$X$ consists of two (terminal) cyclic
quotient singularities~$P$ and $Q$ of types $\frac1r(1,a,-a)$ and
$\frac1r(-1,a,-a)$, respectively. 
The central fiber
$C:=f^{-1}(o)_{\mathrm{red}}$ is irreducible and 
$-K_X\cdot C=2/r$.
The structure morphism $f: X\to Z$ is the
projection to
$Z:=\CC^2/\mumu_r$. 
The singularity of~$Z$ at the origin is of type
\typeDu{A}{r-1} and the discriminant curve is empty.
\end{setup}

Note that all the transformations described in this section can be realized also in terms of toric geometry (cf. \cite{Shramov:toric-links}).

\begin{theorem}
\label{thm:Tr}
Let $f: X\to Z$ be a $\QQ$-conic bundle, which is of type \typeci{T}{r,a}, $0<a<r$ over a point $o\in Z$,
and let $C:=f^{-1}(o)_{\mathrm{red}}$.
Let $p: \tilde X\to X$ be a Kawamata blowup.
Then $p$ can be completed to an md-link
\eqref{eq:link:0}, where 
$\chi$ is a single flip of type  \typeci{k1A}{} in the proper transform $\tilde C\subset \tilde X$ of $C$,
$\al$ is a partial crepant resolution of $o\in Z$, and 
$f'$ is a $\QQ$-conic bundle having at most two singular fibers\footnote{Here and below we regard a $\QQ$-conic bundle of type~\typeci{T}{1} as a  $\PP^1$-bundle.} of types 
\typeci{T}{a,\overline{(r-a)}_a}
and \typeci{T}{r-a,\overline{(a)}_{r-a}}, where $\overline{(x)}_y$ denotes the residue modulo $y$.
\end{theorem}

\begin{proof}
By symmetry we may assume that $p$ is the blowup of $P$.
Since $P\in X$ is a cyclic quotient singularity, there is only one choice for 
its Kawamata blowup: it must be the weighted blowup $p: \tilde{X}\to X$
with weights $\frac 1r (1,a,r-a)$ (see Theorem~\ref{thm:Kblup}).
By Propositions~\ref{prop:link} and~\ref{prop:Q:w-blowup} there exists an md-type link \eqref{eq:link:0}.
The variety $\tilde{X}$ has on $E$ two terminal quotient singularities $\tilde P_1$ and $\tilde P_2$
of types
\begin{equation}
\label{toric-2-singularities}
\textstyle\frac 1a(1,-r,r)\quad\text{ and}\quad \frac 1{r-a}(1,a,-r),
\end{equation}
respectively, 
and $\tilde C$ intersects the exceptional divisor $E\simeq \PP(1,a,r-a)$ transversely at a smooth point.
In particular, $\dif(X)=2r-4$ (see Corollary~\ref{cor:eq:difSho}).
It follows from Corollary~\ref{cor:Kblup:di} that
\begin{equation*}
\dif(\tilde{X})=2r-3.
\end{equation*}
By \eqref{equation-computation-K.C-2} we have 
\[\textstyle
K_{\tilde{X}}\cdot \tilde{C}=K_X\cdot C+\frac1r E\cdot \tilde{C}=-\frac 1{r}
\]
because $E\cdot \tilde C=1$.
Hence $\tilde C$ is a flipping curve.

Recall that
$\al: Z'\to Z$ is a crepant blowup of $o\in Z$ (see Proposition~\xref{prop:link}).
Thus $Z'$ has Du Val singularities of types \typeDu{A}{r'-1} and \typeDu{A}{r''-1}
with $r'+r''=r$ (the cases $r'=1$ and $r''=1$ are not excluded). 
The curve $\Gamma$ is smooth and rational, hence by Corollaries~\ref{cor:cb:plt} and~\ref{cor:cb:plt4} the discriminant divisor of $f'$ is trivial and $f'$ is also of type~\typeb{T} near each 
singular fiber.
More precisely, $f'$ has at most two singular fibers of types~\typeci{T}{r'} and
\typeci{T}{r''}.
In this case,
\begin{equation*}
\dif(X')=2(r'-1)+2(r''-1)=2r-4.
\end{equation*}
Thus $\dif(\tilde{X})=\dif(X')+1$.
By Corollary~\ref{cor:flips} the map $\chi$ is a single flip, hence it is an isomorphism near $\tilde P_1$ and $\tilde P_2$. 
In particular, $X'$ has 
singularities of indices $a$ and $r-a$ (see \eqref{toric-2-singularities}).
This means that up to permutation we have $r'=a$ and $r''=r-a$.
\end{proof}

Note that in the presentation \ref{setup:Sect:T} there is an involution acting on~$X$ fiberwise and interchanging singular points of~$X$.
Then the transformation \eqref{eq:Glink} below can be regarded as $G$-link (Sarkisov link in the category of $G$-varieties \cite{P:G-MMP}),

\begin{theorem}
\label{thm:TrTr}
Let $f: X\to Z$ be a $\QQ$-conic bundle, which is of type \typeci{T}{r,a} over a point $o\in Z$,
and let $C:=f^{-1}(o)_{\mathrm{red}}$ be the fiber with reduced structure.
Take $a$ so that $0<a<r/2$.
Then there exists the following transformation 
\begin{equation}
\label{eq:Glink}
\vcenter{
\xymatrix@C=4em@R=1.3em{
\tilde{X}\ar@{-->}[r]^{\chi'}\ar[d]^p& X^+\ar@{-->}[r]^{\chi''} &X'\ar[d]^{f'}
\\
X\ar[d]^f & & Z'\ar[dll]_{\al}
\\
Z
}}
\end{equation}
where $p$ is the Kawamata blowups
of both singular points, $f'$ is a $\QQ$-conic bundle having at most three singular fibers of types \typeci{T}{a}, \typeci{T}{a}, \typeci{T}{r-2a},
$\chi'$ is an Atiyah-Kulikov flop in the proper transform $\tilde C\subset \tilde X$ of $C$, $\chi''$ is the composition of  flips in two disjoint smooth rational curves,
and $\al$ is a crepant birational
contraction.
\end{theorem}

\begin{proof}
\renewcommand{\qedsymbol}{}
Consider the Kawamata weighted blowups
$p: \tilde{X}\to X$
of both singular points~$P$ and $Q$ with weights $\frac 1r (1,a,r-a)$ and $\frac 1r (1,r-a,a)$, respectively (see Theorem~\ref{thm:Kblup}).
The variety $\tilde{X}$ has at most four terminal quotient singularities:
\begin{itemize}
\item 
two points $P_1$, $Q_1$ of type $\frac 1a(1,-r,r)$, $p(P_1)=P$, $p(Q_1)=Q$;
\item 
two points $P_2$, $Q_2$ of type $\frac 1{r-a}(1,a,-r)$, $p(P_2)=P$, $p(Q_2)=Q$.
\end{itemize}
Let $E_1$ and $E_2$ be $p$-exceptional divisors over~$P$ and $Q$, respectively, and let $E:=E_1+E_2$.
Then $E_1\simeq \PP(1,a, r-a)$ and $E_2\simeq \PP(1,r-a, a)$. Let $x, u,v$ be quasi-homogeneous coordinates on 
these weighted projective planes.
The curve $\tilde C$ intersects each $E_i$ transversely at the point $\{u=v=0\}=(1,0,0)$
which is smooth on $E_i$ and on $\tilde X$.
Then $E_i\cdot \tilde C=1$ and $K_{\tilde{X}}\cdot \tilde C=0$ by \eqref{equation-computation-K.C-2}.
Hence $\tilde C$ is a flopping curve and there exists a flop $\chi':\tilde X \dashrightarrow X^+$ over~$Z$. 
We claim that $\chi'$ is the Atiyah-Kulikov flop. To show this we consider (analytic) surfaces
$F_u=\{u=0\}/\mumu_r$ and $F_v=\{v=0\}/\mumu_r$ containing $C$.
Let $\tilde F_u$ and $\tilde F_v$ be their proper transforms on $\tilde X$.
Since the curve $\tilde C$ does not pass through any of the points $P_1$, $Q_1$, $P_2$, $Q_2$, it 
is contained in the smooth locus of $\tilde F_u$ as well in the smooth locus of $\tilde F_v$. Hence $\tilde C$ is a $(-1)$-curve on $\tilde F_u$
and so the normal bundle $\NNN_{\tilde C/\tilde X}$ contains the subbundle $\NNN_{\tilde C/\tilde F_u}\simeq \OOO_{\PP^1}(-1)$.
This implies that 
\[
\NNN_{\tilde C/\tilde X}\simeq \OOO_{\PP^1}(-1)\oplus \OOO_{\PP^1}(-1)
\]
and so is the Atiyah-Kulikov flop. 
Moreover, $F_u$ and $F_v$ intersect each other transversely at the general point of $C$, hence $\tilde C= \tilde F_u\cap \tilde F_v$
(scheme-theoretically).

Let $E_i^+:=\chi'_*(E_i)$, $F_u^+:=\chi'_*(\tilde F_u)$, $F_v^+:=\chi'_*(\tilde F_v)$, and let $C^+$ be the flopped curve. Then the restriction $\chi'^{-1}_{E_i}: E_i^+\dashrightarrow E_i$ is a 
morphism and
moreover, it is the blowup of the point $(1,0,0)$ with exceptional divisor $C^+$. In particular, $E_1^+\cap E_2^+=C^+$
and $F_u^+\cap F_v^+=\varnothing$. 

Consider the curves $L_1=\{u=0\}\subset E_1$ and $L_2=\{v=0\}\subset E_2$.
Then $\OOO_{E_i}(L_i)=\OOO_{\PP(1,a, r-a)}(a)$, hence $L_i^2=a/(r-a)$.
Clearly, $L_1=E_1\cap \tilde F_u$ and $L_2=E_2\cap \tilde F_v$.
Let $L_i^+\subset E_i^+$ be its proper transform of $L_i$.
Then 
\[
(L_i^+)^2=L_i^2-1=(2a-r)/(r-a)<0
\]
since $L_i$ passes through $(1,0,0)$.
Therefore, the Mori cone $\NE(E_i^+)$ is generated by the classes of curves $L_i^+$ and $C^+$.
Since $F_u^+\cap F_v^+=\varnothing$, the curves $L_1^+$ and $L_2^+$ are disjoint.
Since the fiber of $X^+\to Z$ over~$o$ is $E_1^+\cup E_2^+$, the Mori cone $\NE(X^+/Z)$ is generated by the classes of
$C^+$, $L_1^+$ and $L_2^+$. Thus $L_1^+$ and $L_2^+$ generate $K_{X^+}$-negative extremal rays (because $\uprho (X^+/Z)=3$).
Since $(L_i^+)^2<0$ for $i=1$ and $2$, both $L_1^+$ and $L_2^+$ are flipping curves.
By the construction $L_1^+$ contains $P_2^+:=\chi'(P_2)$ and does not contain $P_1^+:=\chi'(P_1)$.
Similarly, $L_2^+$ contains $Q_2^+:=\chi'(Q_2)$ and does not contain $Q_1^+:=\chi'(Q_1)$.

Running the MMP on $X^+$ over~$Z$ we get three possibilities:
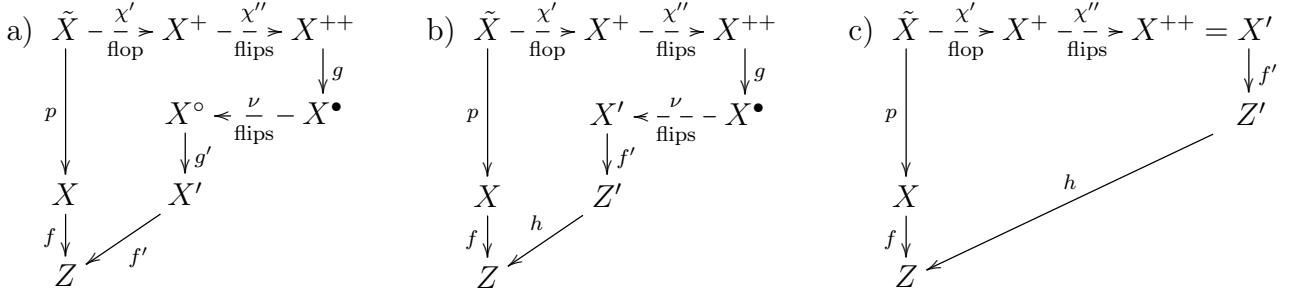
\[
\text{\nN}\label{pic:a}
\xymatrix@R=1.3em{
\tilde X\ar@{-->}[r]^{\chi'}_{\text{flop}}\ar[dd]_p&X^+\ar@{-->}[r]^{\chi''}_{\text{flips}}&X^{++}\ar[d]^{g}
\\
&X^{\circ}\ar[d]^{g'}&X^{\bullet}\ar@{-->}[l]^{\text{flips}}_{\nu}
\\
X\ar[d]_f&X'\ar[dl]^{f'}&
\\
Z
}
\qquad
\text{\nN}
\xymatrix@R=1.3em{
\tilde X\ar@{-->}[r]^{\chi'}_{\text{flop}}\ar[dd]_p&X^+\ar@{-->}[r]^{\chi''}_{\text{flips}}&X^{++}\ar[d]^{g}
\\
&X'\ar[d]^{f'}&X^{\bullet}\ar@{-->}[l]^{\text{flips}}_{\nu}
\\
X\ar[d]_f&Z'\ar[dl]_{\al}&
\\
Z
}
\qquad
\text{\nN}
\xymatrix@R=1.3em{
\tilde X\ar@{-->}[r]^{\chi'}_{\text{flop}}\ar[dd]_p&X^+\ar@{-->}[r]^<(+0.2){\chi''}_<(+0.2){\text{flips}}&X^{++}=X'\ar[]!<1.5em,0em>;[d]!<1.5em,0em>^{f'}
\\
&&\phantom{X^{++}=}Z'\ar []!<1.5em,0em>; [ddll]_{\al}
\\
X\ar[d]_f&&
\\
Z
} 
\]
where $\chi''$ is a composition of at least two flips, $\nu$ is a composition of flips (or isomorphisms), 
$g$ and $g'$ are a divisorial extremal Mori contractions, 
$f'$ is a $\QQ$-conic bundle, and $\al$ is a birational crepant contraction of Du Val surfaces.
We will show that the cases \textup{a)} and \textup{b)} do not occur.

It follows from \cite[Theorem~2.4]{Kollar:flops} and Corollary~\ref{cor:Kblup:di} that
\begin{equation*}
\dif(X^+)=\dif(\tilde{X})=2r-4.
\end{equation*}
Therefore, by Corollary~\ref{cor:flips}
\begin{equation}
\label{Shokurov:diff-compare:TT}
\dif(X^{++})\le \dif(X^+)-2 = 2r-6.
\end{equation}
\end{proof}

\begin{sclaim}
\label{claim:T}
We have $\Delta_{f'}=\varnothing$, i.e. 
the $\QQ$-conic bundle $f'$ is of type \typeci{T}{} near each singular fiber.
\end{sclaim}

\begin{proof}
This is a consequence of   Corollaries~\ref{cor:cb:plt} and~\ref{cor:cb:plt4}.
\end{proof}

\begin{proof}[Proof of Theorem~\xref{thm:TrTr} \textup(continued\textup)]
Now consider the cases a), b), c) separately.

\subsection*{Case a)}
Then morphisms $g$ and $g'$ contract the proper transforms of $p$-exceptional divisors $E_1$ and $E_2$.
Hence the map $X\dashrightarrow X'$ is an isomorphism in codimension $1$.
Since $-K_X$ and $-K_{X'}$ are ample over~$Z$, the map $X\dashrightarrow X'$ is, in fact, an isomorphism.
Then by Corollaries~\ref{cor:Kblup:di} and~\ref{cor:flips}
we have
\[
\dif(X^{++})\ge \dif(X^{\bullet})-1\ge \dif(X^{\circ})-1\ge \dif(X')-2= \dif(X)-2=2r-4.
\]
This contradicts \eqref{Shokurov:diff-compare:TT}.

\subsection*{Case b)}
Then $\al: Z'\to Z$ is a partial resolution of the \typeDu{A}{r-1}-point
$o\in Z$. Let $\Gamma$ be the $\al$-exceptional divisor. Since $\uprho(Z'/Z)=1$, $\Gamma$ is irreducible.
The surface $Z'$ has Du Val singularities of types \typeDu{A}{r'-1} and \typeDu{A}{r''-1}
with $r'+r''=r$ (the cases $r'=1$ and $r''=1$ are not excluded). 
By Claim~\ref{claim:T} the discriminant divisor of $f'$ is trivial and $f'$ is also of type~\typeb{T}
near each singular fiber.
More precisely, $f'$ has two (or one) singular fibers of types~\typeci{T}{r'} and
\typeci{T}{r''}. Then, as in the case a), we have
\[
\dif(X^{++})\ge \dif(X^{\bullet})-1\ge \dif(X')-1=2(r'-1)+2(r''-1)=2r-4.
\]
This again contradicts \eqref{Shokurov:diff-compare:TT}.

\subsection*{Case c)}
Then $\al: Z'\to Z$ is a partial resolution of the \typeDu{A}{r-1}-point
$o\in Z$. Let $\Gamma$ be the $\al$-exceptional divisor. Since $\uprho(Z'/Z)=2$, $\Gamma$ has exactly two components $\Gamma_1$ and $\Gamma_2$.
The surface $Z'$ has Du Val singularities $o'$, $o''$, $o'''$ of types \typeDu{A}{r'-1}, \typeDu{A}{r''-1}, and \typeDu{A}{r'''-1},
respectively,
with $r'+r''+r'''=r$ (the cases $r'=1$, $r''=1$, and $r'''=1$ are not excluded). 
We may assume that $o'\in \Gamma_1$, $\Gamma_1\cap \Gamma_2=\{o''\}$, and $o'''\in \Gamma_2$.
By  $f'$ is also of type~\typeb{T}
near each singular fiber.
More precisely, $f'$ has three (or less) germs of types~\typeci{T}{r'}, \typeci{T}{r''}, and
\typeci{T}{r'''}. Then by Corollary~\ref{cor:eq:difSho} we have
\[
\dif(X')=2r'-2+2r''-2+2r'''-2=2r-6.
\]
Comparing this with \eqref{Shokurov:diff-compare:TT} we obtain $\dif(X^{++})= \dif(X^+)-2$.
Thus $\chi''$ is the composition of exactly two flips in $L_1^+$ and $L_2^+$ (see Corollary~\ref{cor:flips}).
Up to permutation we may assume that $E_i'=f^{-1}(\Gamma_i)$. 
Since the fibers of $f'$ are irreducible, we have $E_1'\cap E_2'=C'$ and $C'$ is the fiber over $o''$.
Recall that the points $P_1^+\in E_1^+$ and $Q_2^+\in E_2^+$ of type $\frac 1a(1,-r,r)$ are not contained in the 
$\chi''$ flipping locus. Therefore, $P_1':=\chi''(P_1^+)$ and $Q_1':=\chi''(Q_1^+)$
are also of type $\frac 1a(1,-r,r)$. Moreover, $P_1',\, Q_1'\notin C'$. Therefore, $f'(P_1')=o'$ and $f'(P_2')=o'''$.
This implies that $r'=r'''=a$ and $r''=r-2a$.
\end{proof}

\section{$\QQ$-conic bundles of type \typeci{k2A}{}}
\label{Sect:k2}
In this section we adopt the following setup.

\begin{setup}
Let $f: X\to Z$ be a $\QQ$-conic bundle, which is of type \typeci{k2A}{r} over a point $o\in Z$
and is a standard conic bundle outside~$o$. Let $C:=f^{-1}(o)_{\mathrm{red}}$ be the central fiber.
Recall that in this situation
the tubular neighborhood $(X\supset C)$ is biholomorphic to the quotient of the hypersurface
\[
\{ x_1^2+ux_2^2+vx_3^2=0\}\subset\PP^2_{x_1,x_2,x_3}\times\CC^2_{u,v}
\]
by the $\mumu_{r}$-action with weights
\[
\wt(x_1,x_2,x_3;\, u,v)=(a, -1,0;\,1,-1),
\]
where $r=2a+1$ and $f: X\to Z$ is the
projection to
$Z:=\CC^2/\mumu_r$.
The singular locus of~$X$ consists of two (terminal) cyclic quotient
singularities~$P$ and $Q$ of types $\frac 1r(a+1,1,-1)$ and $\frac 1r(a,-1,1)$, respectively.
The singularity of~$Z$ at the origin is of type~ \typeDu{A}{r-1}.
For the singular fiber
$C:=f^{-1}(o)_{\mathrm{red}}$ we have
$-K_X\cdot C=1/r$.
The discriminant curve $\Delta_f$ is given by $\{uv=0\}/\mumu_r$.
In particular, the pair $(Z,\Delta_f)$ is log canonical.
\end{setup}

\begin{theorem}
\label{thm:k2A}
Let $f: X\to Z$ be a $\QQ$-conic bundle, which is of type \typeci{k2A}{r} over a point $o\in Z$
and is a standard conic bundle outside~$o$.
Let $p: \tilde X\to X$ be a Kawamata blowup.
Then $p$ can be completed to an md-link
\eqref{eq:link:0}, where 
$\chi$ is a composition of flips,
$\al$ is a partial crepant resolution of $o\in Z$, and
$f'$ is a $\QQ$-conic bundle having two singular fibers\footnote{Here we regard a $\QQ$-conic bundle of type~\typeci{k2A}{1} as a standard 
conic bundle.} of types~\typeci{k2A}{r-2}
and~\typek{ID}{2}. The discriminant curve $\Delta_{f'}$ is the set-theoretic preimage of $\Delta_f$. 
\end{theorem}

\begin{proof}
By symmetry we may assume that $p$ is the blowup of $P$.
As in the case \typeci{T}{} there is only one choice for 
its Kawamata blowup: it must be the weighted blowup $p: \tilde{X}\to X$
with weights $\frac1r (1,2,r-2)$ (see Theorem~\ref{thm:Kblup}).
Then $E\cap \tilde C$ is a singularity of type $\frac 12(1,1,1)$ on $\tilde{X}$ and $E\cdot \tilde C=1/2$.
By \eqref{equation-computation-K.C-2} we have 
\[\textstyle
K_{\tilde{X}}\cdot \tilde{C}=K_X\cdot C+\frac1r E\cdot \tilde{C}=-\frac 1{2r}.
\]
By Proposition~\ref{prop:link}
we obtain an md-type Sarkisov link, where $f'$ is a $\QQ$-conic bundle
and $\al: Z'\to Z$ is a crepant contraction. 
For the discriminant curve $\Delta_{f'}$ we have
$\Delta_{f'}\subset \Supp(\al^{*}(\Delta_f))$. 
Let $\Delta'\subset Z'$ be the proper transform of $\Delta_f$.
Since the pair $(Z,\Delta_f)$ is lc, so is the pair 
$(Z', \al^{*}\Delta_f)$ and $\al^{*}\Delta_f= \Delta'+\Gamma$,
where $\Gamma$ is the $\al$-exceptional curve. Hence, $(Z,\Delta')$ is plt.
In particular, $\Delta'$ is smooth in a neighborhood of $\Gamma$.
On the other hand, the surface $Z'$ is singular because $r\ge 3$. Hence $f'$ is not a standard conic bundle 
and $\Delta_{f'}\neq \Delta'$ by Corollary~\xref{cor:cb:plt}.
Therefore, $\Delta_{f'}=\al^{*}\Delta_f= \Delta'+\Gamma$.
Since $\Delta_f=\{uv=0\}/\mumu_r$ has two analytic branches at $o$ and $(Z',\Delta'+\Gamma)$ is lc, 
the intersection $\Gamma\cap \Delta'$ consists of exactly two points, say $o'$ and $o''$,
so that  $(Z',\Delta'+\Gamma)$ is plt outside $o'$ and $o''$.
This implies that 
$f'$ is a standard conic bundle outside $o'$ and $o''$ (see Corollary~\xref{cor:cb:plt}). 
Further, as in the case~\typeb{T}, we see that $Z'$ has at most two Du Val points of types \typeDu{A}{r'-1}
and \typeDu{A}{r''-1} with $r'+r''=r$ and $r',\, r''\ge 1$. Since $r$ is odd, we may assume that $r'$ is odd and $r''$ is even.
Then $f'$ over $o''$ must be of type~\typek{ID}2 by Corollary~\xref{cor:cb:plt}\ref{cor:cb:plt3}. Hence $r''=2$ and $r'=r-2$.
If $r'=1$, then $f'$ is a standard conic bundle over $o'$. Otherwise $f'$ is of type~\typeci{k2A}{r'} over $o'$
(again by Corollary~\xref{cor:cb:plt}\ref{cor:cb:plt3}).
\end{proof}

Note that a $\QQ$-conic bundle germ of type \typeci{k2A}{} can be regarded as $G$-variety: 
there exists an involution interchanging singular points.
Then according to general philosophy (see e.g.~\cite{P:G-MMP}) there should exists $G$-equivariant link similar to 
Theorem~\ref{thm:TrTr}. However in our situation this link must start with Kawamata blowups 
of both non-Gorenstein points and the divisor $-K_{\tilde X}$ is not nef over~$Z$.
This means that the link involves antiflips, which are more difficult to analyze. 

\section{$\QQ$-conic bundles of type \typeb{IE^\vee}}
\label{Sect:IE}

In this section we adopt the following setup.
\begin{setup}
Let $f: X\to Z$ be a $\QQ$-conic bundle, which is of type \typeb{IE^\vee} over a point $o\in Z$
and is a standard conic bundle outside~$o$. Let $C:=f^{-1}(o)_{\mathrm{red}}$ be the central fiber.
Recall that in this situation
the tubular neighborhood $(X\supset C)$ is biholomorphic to the $\mumu_4$-quotient of
a complete intersection of two hypersurfaces in $\PP(1^3,2)\times \CC^2$ as described in Table~\ref{tab:sing}.
Moreover, applying analytic coordinate changes 
we may assume that $(X\supset C)$
is the quotient of the following variety in $\PP(1^3,2)\times \CC^2$
\begin{equation}
\label{eq:index2-IE}
\left\{
\begin{array}{lll}
x_1^2+x_2^2+\psi x_3^2+u\,x_4&=&0
\\[7pt]
2\,x_1x_2+x_3^2+\xi_{1,1}x_1^2+\xi_{2,2}x_2^2+2\,\xi_{1,3}x_1x_3+2\,\xi_{2,3}x_2x_3+v\,x_4&=&0
\end{array}
\right.
\end{equation}
by an action of $\mumu_4$ with weights $(3,1,2,1;1,3)$, where 
$\psi,\,\xi_{1,1},\, \xi_{2,2} \in \mathfrak{m}^2\subset \CC\{u,\, v\}$ are semi-invariants of $\mumu_4$-weight $2$
and $\xi_{1,3},\,\xi_{2,3} \in \mathfrak{m}$ are semi-invariants of $\mumu_4$-weights $-1$ and $1$, respectively.
It is easy to see that~$X$ has a unique singular point~$P$ that is a cyclic quotient of type $\frac18(5,1,3)$.
The base point $o\in Z$ has type \typeDu{A}3, hence
$-K_X\cdot C=1/2$ (see \cite[Corollary~4.7.4]{MP:1pt}). 
\end{setup}

\begin{sclaim}[Local description]
It follows from \eqref{eq:index2-IE} that up to analytic isomorphism in a neighborhood of~$P$ we have
the following identification
\[
(C\subset X)=\big(\{x_1^2+x_2^2=2x_1x_2+x_3^2=0\}\subset \CC^3\big) /\mumu_8(5,1,3).
\]
\end{sclaim}

\begin{sclaim}
\label{claim:IFvee:graph}
The discriminant curve $\Delta_f$ has  at~$o$ two smooth analytic branches $\Delta_1$ and $\Delta_2$
such that $(Z, \Delta_1)$ is plt and $(Z, \Delta_2)$ is lc but not plt. Therefore, 
the dual graph of the minimal resolution
of $Z\supset \Delta_f\ni o$ has the following form
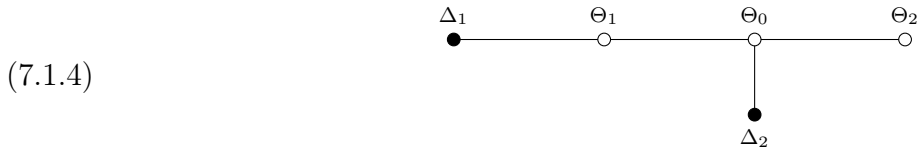
\begin{equation}
\label{eq:IFvee:graph}
\vcenter{\hbox{
\def\sizec{0.2em}
\def\sizev{0.5}
\begin{tikzpicture}
\coordinate(1) at (0,0);
\coordinate(2) at (-2,0);
\coordinate(3) at (2,0);
\coordinate(4) at (4,0);
\coordinate(5) at (2,-1);

\path (1) edge (3);
\path (2) edge (1);
\path (3) edge (4);
\path (3) edge (5);

\path[fill=white,draw=black] (1) circle (\sizec)node[above, yshift=1.5]{$\scriptstyle \Theta_1$};
\path[fill=black,draw=black] (2) circle (\sizec)node[above, yshift=1.5]{$\scriptstyle \Delta_1$};
\path[fill=white,draw=black] (3) circle (\sizec)node[above, yshift=1.5]{$\scriptstyle \Theta_0$};
\path[fill=white,draw=black] (4) circle (\sizec)node[above, yshift=1.5]{$\scriptstyle \Theta_2$};
\path[fill=black,draw=black] (5) circle (\sizec)node[below, yshift=-1.5]{$\scriptstyle \Delta_2$};
\end{tikzpicture}
}}
\end{equation} 
where the vertices $\circ$ correspond to \textup(crepant\textup) exceptional divisors over~$o$ and the vertices
$\bullet$ correspond to the components of $\Delta_f$.
\end{sclaim}

\begin{proof}
It follows from Proposition~\ref{prop:comp-discrim} and \eqref{eq:index2-IE} that the discriminant curve has the form
\[
\Delta_f=\left\{
u (u^2-v^2) +(\text{terms of degree $\ge 5$}) \right\}
\subset Z=\CC^2/\mumu_4(1,-1).
\]
Then it is easy to see that $\Delta_f$ has two smooth analytic branches:
\[
\Delta_1= \left\{ u  +(\text{terms of degree $\ge 2$}) \right\}/\mumu_4,
\quad
\Delta_2= \left\{ u^2-v^2 +(\text{terms of degree $\ge 3$}) \right\}/\mumu_4.
\]
Hence  $(Z, \Delta_1)$ is plt, $(Z, \Delta_2)$ is lc,
and the configuration of curves on the minimal resolution has the desired form
(see \cite[Theorem~9.6]{Kawamata:crep}).
\end{proof}

\begin{theorem}
Let $f: X\to Z$ be a $\QQ$-conic bundle, which is of type \typeb{IE^\vee} over a point $o\in Z$.
Let $p: \tilde X\to X$ be the Kawamata blowup of the unique singular point $P\in X$.
Then $p$ can be completed to an md-link \eqref{eq:link:0},
where
$\chi$ is a composition of flips,
$\al$ is a partial crepant resolution of $o\in Z$,
the $\QQ$-conic bundle $f'$ has a degenerate fiber type \typeci{k2A}3
over a point $o'\in \Gamma=\al^{-1}(o)$ and it is a standard conic bundle outside $o'$. Moreover, the surface $Z'$ 
has a unique singularity $o'$ which is of type \typeDu{A}2 and
the discriminant curve $\Delta_{f'}$ is the proper transform of~$\Delta_f$.
\end{theorem}

\begin{proof}
In our case the extraction $p: \tilde X\to X$ must be the $\frac 18 (5,1,3)$-weighted blowup of~$P$ (see Theorem~\ref{thm:Kblup}).
Then $\tilde{X}$ has two cyclic quotient singularities $\tilde{P}_1$ and $\tilde{P}_2$ of types $\frac 15(-3,1,3)$
and $\frac 13(2,1,1)$, respectively. The curve $\tilde C$ passes through $\tilde{P}_1$ and in the corresponding affine chart we have
\[
\tilde C= \{\tilde{x}_1+\tilde{x}_2^2=2\tilde{x}_2+\tilde{x}_3^2=0\}/\mumu_5(-3,1,3),\qquad E=\{\tilde{x}_1=0\}/\mumu_5(-3,1,3).
\]
Therefore, $E\cdot \tilde C=4/5$ and $K_{\tilde{X}}\cdot \tilde C=-2/5$
by \eqref{equation-computation-K.C-2}. 
According to Theorem~\ref{thm:link:exist}\ref{thm:link:exist2} we obtain an md-link \eqref{eq:link:0}. Moreover, the transformation $\chi$ starts with 
the flip along $\tilde C$.
Recall that $\al: Z'\to Z$ must be a crepant extraction with irreducible exceptional divisor $\Gamma$.
Let $\Delta_i'\subset Z'$ is the proper transform of $\Delta_i$ for $i\in \{1,\, 2\}$ (see Claim~\ref{claim:IFvee:graph}).
Clearly, $\Gamma$ corresponds to one of the vertices $\Theta_0$, 
$\Theta_1$ or~$\Theta_2$ in \eqref{eq:IFvee:graph}.

Assume that $\Gamma$ corresponds to $\Theta_0$. Then
$Z'$ has two points, say $o'$ and $o''$, of type~ \typeDu{A}1 and we may assume $o_i$ corresponds to $\Theta_i$ for $i\in \{1,\, 2\}$.
In this case, $\Delta_1'\cap \Gamma=\{o'\}$ and $\Delta_2'\cap \Gamma$ is a smooth point of $Z'$.
Moreover, $(Z, \Delta_1')$ is plt at $o'$. By Corollary~\xref{cor:cb:plt} we have $\Gamma\subset \Delta_{f'}$.
Hence $o''\in\Delta_{f'} $ and the pair $(Z',\Delta_{f'})$ is plt at $o''$.
This is impossible again by Corollary~\xref{cor:cb:plt}.

Thus $\Gamma$ corresponds to either $\Theta_1$ or $\Theta_2$ 
and $Z'$ has a unique singular point~$o'$, which must be of type~ \typeDu{A}2.
By the classification (see Table~\xref{tab:sing}) the germ of $X'$ over $o'$ is of type \typeci{k2A}3.
In this case, the pair $(Z',\Delta_{f'})$ is lc but not plt at~$o'$.
In particular, $\Delta_{f'}$ has two analytic components passing through $o'$.
If $\Gamma$ corresponds to $\Theta_1$, then $\Gamma\subset \Delta_{f'}$.
But in this case the pair $(Z',\Delta_2'+\Gamma)$ is not lc at~$o'$, a contradiction.
Therefore, $\Gamma$ corresponds to $\Theta_2$ and $\Gamma\not\subset \Delta_{f'}$ by Corollary~\ref{cor:cb:plt4}.
This implies that $f'$ is a standard conic bundle outside $o'$ (see Proposition~\ref{prop:cb:uv})
and so $X'$ has exactly two singular points and they are of type $\frac13(1,1,2)$.
\end{proof}

\section{$\QQ$-conic bundles of type \typeb{IA^\vee{+}IA^\vee}}
\label{Sect:IAIA}
In this section
we consider the following setup.
\begin{setup}
Let $f: X\to Z$ be a $\QQ$-conic bundle, which is of type \typeb{IA^\vee{+}IA^\vee} over a point $o\in Z$
and it is a standard conic bundle outside~$o$. Let $C:=f^{-1}(o)_{\mathrm{red}}$ be the central fiber.
Recall that in this situation
the tubular neighborhood $(X\supset C)$ is biholomorphic to a $\mumu_2$-quotient of
a complete intersection of two hypersurfaces in $\PP(1^3,2)\times \CC^2$ as described in Table~\ref{tab:sing}.
Moreover, applying analytic coordinate changes 
we may assume that $(X\supset C)$
is the quotient of the following variety in $\PP(1^3,2)\times \CC^2$ 
\begin{equation}
\label{eq:IA+IA}
\begin{cases}
x_1^2+x_3^2+\phi x_2^2+u x_4=0
\\
x_2^2+x_3^2+ \xi_{1,1}x_1^2+2\,\xi_{1,2}x_1x_2+2\,\xi_{1,3}x_1x_3+2\,\xi_{2,3}x_2x_3 +v x_4=0
\end{cases}
\end{equation} 
by $\mumu_2(1,1,0,1; 1,1)$, where
$\phi,\,\xi_{1,1},\, \xi_{1,2} \in \mathfrak{m}^2\subset \CC\{u,\, v\}$ are invariants of $\mumu_2$
and $\xi_{1,3},\,\xi_{2,3} \in \mathfrak{m}$ are semi-invariants of $\mumu_2$-weight $1$.
The variety~$X$ has a unique singular point~$P$ that is a cyclic quotient of index $4$ and
the central fiber $C$ has two components $C_1$ and $C_2$ meeting at~$P$.
The point $o\in Z$ is of type \typeDu{A}1 and
$-K_X\cdot C_i=1/2$.
\end{setup}

\begin{sclaim}[Local description]
It follows from \eqref{eq:IA+IA} that up to analytic isomorphism in a neighborhood of~$P$
the pair $C\subset X$ has the following form
\[
\big(\{x_1^2+x_3^2=x_2^2+x_3^2=0\}\subset \CC^3\big) /\mumu_4(1,1,3).
\]
\end{sclaim}

\begin{sclaim}
\label{claim:IA+IA}
The discriminant curve $\Delta_f$ has at $o$ exactly three analytic branches
$\Delta_1$, $\Delta_2$, $\Delta_3$. They are smooth and their proper transforms on the minimal resolution do not meet each other.
\end{sclaim}

\begin{proof}
It follows from Proposition~\ref{prop:comp-discrim} and \eqref{eq:IA+IA} that
\[
\Delta_f=\left\{
uv(u-v) +(\text{terms of degree $\ge 5$})=0\right\}/\mumu_2.\qedhere
\]
\end{proof}

\begin{theorem}
Let $f: X\to Z$ be a $\QQ$-conic bundle, which is of type \typeb{IA^\vee{+}IA^\vee} over a point $o\in Z$
and is a standard conic bundle outside~$o$.
Let $p: \tilde X\to X$ be the Kawamata blowup of the unique singular point $P\in X$.
Then $p$ can be completed to an md-link \eqref{eq:link:0},
where 
$\chi$ is the flip in the proper transform $\tilde C\subset \tilde X$ of the central fiber $C:=f^{-1}(o)_{\mathrm{red}}$,
$\al$ is the minimal resolution of $o\in Z$, the singular locus of $X'$ consists of
one ordinary double point, and the discriminant curve $\Delta_{f'}$ is
the set-theoretic preimage of $\Delta_f$.
\end{theorem}

\begin{proof}
In our case the extraction $p: \tilde X\to X$ must be the $\frac 14 (1,1,3)$-weighted blowup of~$P$ (see Theorem~\ref{thm:Kblup}).
Then $\tilde{X}$ has a unique singular point $\tilde{P}$ which is a 
cyclic quotient of index $3$. 
Both components $\tilde C_1,\, \tilde C_2\subset \tilde C$ passes through $\tilde{P}$ and in the corresponding affine chart we have
\[
\tilde C= \{\tilde{x}^2_1+\tilde{x}_3=\tilde{x}^2_2+\tilde{x}_3=0\}/\mumu_3(1,1,2),\qquad E=\{\tilde{x}_3=0\}/\mumu_3(1,2,1).
\]
Therefore, $E\cdot \tilde C_i=2/3$ and $-K_{\tilde{X}}\cdot \tilde C_i=1/3$
by \eqref{equation-computation-K.C-2}. By Proposition~\ref{prop:link} we obtain an md-type link,
where $\al: Z'\to Z$ is the minimal resolution. Moreover, the transformation $\chi$ starts with 
the flip along $\tilde C$. Since $\dif(\tilde X)=2$, by Proposition~\ref{prop:flip:2comp} the map $\chi$ is a single flip and 
$X'$ is Gorenstein. 
Hence the surface $E'$ is Gorenstein as well (because it is a Cartier divisor in $X'$).
Let $C'\subset X'$ be the flipped curve.
Since $(K_{\tilde X}+E)\cdot \tilde C=2/3>0$, we have $(K_{X'}+E')\cdot C'<0$.
Note also that $K_{X'}+E'$ is negative on the fibers of $f'$.
Since $\uprho(X'/Z)=2$, the divisor $-(K_{X'}+E')$ is ample over~$Z$.
By the adjunction $-K_{E'}$ is ample.
Therefore, $E'$ is a Gorenstein (singular) del Pezzo surface.

Assume that $E'$ is normal. Then by Zariski's main theorem the map 
$\chi^{-1}_E:E'\dashrightarrow E=\PP(1,1,3)$
is a birational morphism. In this case, the singularities of $E'$ must be rational, hence they are Du Val
because $E'$ is a Gorenstein.
Since $-K_{E'}$ is ample, so are the singularities of $E$, a contradiction.
Therefore, $E'$ is not normal. The non-normal locus is contained in $C'$, hence it is $f'$-horizontal
and irreducible.
Let $\nu: E''\to E'$ be the normalization of 
$E'$. Then $E''$ is a smooth ruled surface, $E''\simeq \FF_3$ and the induced morphism $E''\to E\simeq \PP(1,1,3)$
is the contraction of the negative section~$\Sigma$ (cf. \cite[Theorem~1.1]{Reid:dP94}).
Thus $C'=\nu(\Sigma)$. In particular, $C'$ is irreducible and coincides with the non-normal locus of~$E'$.
By Proposition~\ref{prop:flip:2comp} the variety $X'$ is singular and its singular locus consists of 
one ordinary double point.
Finally, by Claim~\ref{claim:IA+IA} the discriminant curve $\Delta_f$ has three smooth analytic branches $\Delta_1$, $\Delta_2$, $\Delta_3$ at~$o$ so that
their proper transforms $\Delta_1', \Delta_2', \Delta_3'\subset Z'$ are disjoint.
Since $E'$ is not normal, the curve $\Gamma=f'(E')$ is contained in the discriminant.
Thus, $\Delta_{f'}=\Gamma+\Delta_1'+\Delta_2'+\Delta_3'$.
\end{proof}

\section{$\QQ$-conic bundles of type \typeb{IA^\vee}}
\label{Sect:IA}

In this section we adopt the following setup.
\begin{setup}
Let $f: X\to Z$ be a $\QQ$-conic bundle, which is of type \typeb{IA^\vee} over a point $o\in Z$
and is a standard conic bundle outside~$o$. Let $C:=f^{-1}(o)_{\mathrm{red}}$ be the central fiber.
Recall that in this situation
the tubular neighborhood $(X\supset C)$ is biholomorphic to the $\mumu_2$-quotient of
a complete intersection of two hypersurfaces in $\PP(1^3,2)\times \CC^2$ as described in Table~\ref{tab:sing}.
Moreover, applying analytic coordinate changes 
we may assume that $(X\supset C)$
is the quotient of the following variety in $\PP(1^3,2)\times \CC^2$
\begin{equation}
\label{eq:IAdual}
\begin{cases}
x_1^2+x_2^2+\phi x_3^2+ux_4 =0
\\
x_3^2+ \xi_{1,1}x_1^2+\xi_{2,2}x_2^2+2 \xi_{1,2}x_1x_2+2 \xi_{1,3}x_1x_3+2 \xi_{2,3}x_2x_3 +vx_4 =0
\end{cases}
\end{equation}
by $\mumu_2(0,1,0,1; 1,1)$, where $\phi, \xi_{1,1}, \xi_{2,2}, \xi_{1,3}\in \mathfrak{m}^2\subset \CC\{u,\, v\}$ are invariants and
$\xi_{1,2},\, \xi_{2,3} \in \mathfrak{m}$
are semi-invariants of $\mumu_2$.
The point $o\in Z$ is of type \typeDu{A}1, hence 
$-K_X\cdot C=1/2$ (see \cite[Corollary~4.7.4]{MP:1pt}).
The variety~$X$ has a unique non-Gorenstein point $P$, which is a cyclic quotient of index $4$. 
\end{setup}

\begin{sclaim}
\label{claimIA:sing}
If $\xi_{1,2}$ contains the term $u$, then $P$ is the only singular point of~$X$.
Otherwise~$X$ has a Gorenstein singular point. 
\end{sclaim}

\begin{proof}
Follows from direct computations. 
\end{proof}

\begin{sclaim}[Local description]
It follows from~\eqref{eq:IAdual} that near
the point $P\in X$ there is the following identification
\[
(C\subset X) =\big(\{x_1^2+x_2^2=x_3=0\}\subset \CC^3\big) /\mumu_4(1,3,1).
\]
\end{sclaim}

\begin{sclaim}
\label{claimIA:discr}
In the above notation for the discriminant curve there is a decomposition 
\[
\Delta_f=\Delta_1+\Delta_2, 
\]
where 
$\Delta_1$ is smooth and irreducible and $\Delta_2$ is singular \textup(and possibly analytically reducible\textup).
Moreover,
the proper transforms $\Delta_1', \Delta_2'\subset Z'$ of $\Delta_1$ and $\Delta_2$ on the minimal resolution $\al: Z'\to Z$ are disjoint, 
$\Delta_1'$ meets the exceptional divisor $\Gamma$ transversely, $\Delta_2'\cap \Gamma$ is a single point,
and $\Delta_2'\cdot \Gamma=2$.
Moreover, $\Delta_2'$ is smooth if and only if $P$ is the only singular point of~$X$.
\end{sclaim}

\begin{proof}
It follows from Proposition~\ref{prop:comp-discrim} and \eqref{eq:IAdual} that 
\[
\Delta_f=\left\{
u v^2
-v^3\phi +u^2v( \xi_{2,3}^2 - \xi_{1,1} - \xi_{2,2}) -u^3\xi_{1,2}^2 
+(\text{terms of degree $\ge 7$})
=0
\right\}/\mumu_2(1,1).
\]
Hence we have an analytic decomposition $\Delta_f=\Delta_1+\Delta_2$, where 
\[
\Delta_1=\{u+(\text{terms of degree $\ge 3$})=0 \}/\mumu_2 \quad \text{and}\quad 
\Delta_2=\{v^2+(\text{terms of degree $\ge 4$})=0 \}/\mumu_2.
\]
In our situation, the minimal resolution $\al: Z'\to Z$ is the usual blowup of the origin.
Clearly, we have $\Delta_1'\cap \Delta_2'=\varnothing$.
In the chart  $U_u=\{u\neq 0\}$ we have $\Gamma=\{u=0\}/\mumu_2$, $\Delta_1'\cap U_u=\varnothing$, and 
\[
\Delta_2'=\left\{
v^2-v^3\phi(u,uv) +v\big( \xi_{2,3}(u,uv)^2 - \xi_{1,1}(u,uv) - \xi_{2,2}(u,uv)\big) -\xi_{1,2}(u,uv)^2 
+\cdots
\right\}/\mumu_2(1,0).
\]
Hence $\Delta_2'\cdot \Gamma=2$.
Moreover, $\Delta_2'$ is smooth if and only if its equation contains the term $u^2$.
\end{proof}

\begin{theorem}
Let $f: X\to Z$ be a $\QQ$-conic bundle, which is of type \typeb{IA^\vee} over a point $o\in Z$
and is a standard conic bundle outside~$o$.
Let $p: \tilde X\to X$ be the Kawamata blowup of the unique non-Gorenstein point $P\in X$.
Then $p$ can be completed to an md-link \eqref{eq:link:0},
where 
$\chi$ is the flip in the proper transform $\tilde C\subset \tilde X$ of the central fiber $C:=f^{-1}(o)_{\mathrm{red}}$,
$\al$ is the minimal resolution of $o\in Z$, and the discriminant curve $\Delta_{f'}$ is
the set-theoretic preimage of $\Delta_f$.
Moreover, if $P$ is the only singular point of~$X$, then $X'$ is Gorenstein.
\end{theorem}

\begin{proof}
As in the case \typeb{IA^\vee{+}IA^\vee}, the extraction $p: \tilde X\to X$ must be the $\frac 14 (1,3,1)$-weighted blowup of~$P$ (see Theorem~\ref{thm:Kblup}).
Then $\tilde{X}$ has a unique non-Gorenstein point, which is 
a cyclic quotient singularity $\tilde{P}$ of index $3$. 
The curve $\tilde C$ passes through $\tilde{P}$ and in the corresponding affine chart we have
\[
\tilde C= \{\tilde{x}^2_1+\tilde{x}_2=\tilde{x}_3=0\}/\mumu_3(1,2,1),\qquad E=\{\tilde{x}_2=0\}/\mumu_3(1,2,1).
\]
Therefore, $E\cdot \tilde C=2/3$ and $K_{\tilde{X}}\cdot \tilde C=-1/3$
by \eqref{equation-computation-K.C-2}. By Proposition~\ref{prop:link} we obtain an md-type link,
where $\al: Z'\to Z$ is the minimal resolution. Moreover, the transformation $\chi$ starts with 
a flip along $\tilde C$.
We have $\dif(\tilde{X})=2$, hence $\dif(X')\le 1$.
Denote $o_1:=\Delta_1'\cap \Gamma$ and $o_2:=\Delta_2'\cap \Gamma$ (see Claim~\ref{claimIA:discr}).

We claim that $\chi$ is a single flip in $\tilde C$.
Assume the converse. Since $\dif(\tilde X)=2$, the map $\chi$ must be a composition of two flips (see \eqref{eq:flip}):
\[
\chi: \tilde X \overset{\chi_1} \dashrightarrow X^+ \overset{\chi_2}\dashrightarrow X',
\]
where $\dif(X_2)=1$ and $\dif(X')=0$. Thus 
$X^+$ is of index $2$ and $X'$ is Gorenstein.
Let $C^+\subset X^+$ (resp. $L^+\subset X^+$) be the $\chi_1$-flipped (resp. $\chi_2$-flipping) curve,
let $C':=\chi_{2 *}C^+$, and let $L'\subset X'$ be the $\chi_2$-flipped curve.
By our construction, $\tilde X$ is smooth outside $\tilde C$, hence $X^+$ is smooth outside $C^+$
and $X'$ is smooth outside $C'\cup L'$.
If $X'$ has a singular point, say $Q'$, then there are at least three exceptional divisors over $C'\cup L'\subset X'$ 
of discrepancy $1$: these are divisors with centers $C'$, $L'$,
and $Q'$ \cite{Markushevich96:discr}. But then $\dif(\tilde X)\ge 3$ by Theorem~\ref{thm:flips}, a contradiction.
Therefore, $X'$ is smooth and $f'$ is a standard conic bundle.
In particular, $\Delta_{f'}$ is a normal crossing curve.
Since $\Delta_2'\cdot \Gamma=2$ and $\Delta_2'\cap \Gamma$ is a single point, we have $\Gamma\not\subset \Delta_{f'}$. 
Since $\Gamma$ is a smooth curve, the surface $E'=f'^{-1}(\Gamma)$ is normal (see Lemma~\ref{lemma:St-c-b}). 
We claim that the surface $E^+:=\chi_{1 *}E$ is normal as well. Indeed, otherwise $E^+$ must be singular along the flipped curve
$\tilde C^+$. On the other hand, $\chi_2$ is an isomorphism at the general point of $\tilde C^+$,
hence $E'$ is also singular along the image of $\tilde C^+$, a contradiction. Thus $E^+$ is normal.
By Zariski's main theorem the map $\chi_1^{-1}|_{E^+}:E^+\dashrightarrow E=\PP(1,1,3)$
is a morphism whose exceptional divisor coincides with $\tilde C^+$. Hence, $\uprho(E^+)=2$.
Note that $E^+$ is the whole fiber over~$o$ of $X^+\to Z$. In this situation the flipping curve
$L^+\subset X^+$ must be contained in $E^+$. Clearly, the flipped curve $L'\subset X'$ is contained in $E'=f'^{-1}(\Gamma)$.
Thus $E^+\setminus L\simeq E'\setminus L'$, where the curves $L$ an $L'$ are irreducible (see \cite[Theorem~4.2]{KM:92}).
This shows that $\uprho(E')=2$.
On the other hand, 
the fiber of $f'_{E'}: E'\to \Gamma$ over $\Delta_1'\cap \Gamma$ is reducible and so $\uprho(E')\ge 3$, a contradiction.

Therefore,
$\chi$ is a single flip.
Consider the case where $E'$ is normal. Then the curve $\Gamma=f'(E')$ is not contained in the discriminant
and $f'$ is a standard conic bundle over $Z' \setminus \{o_2\}$ (see Proposition ~\ref{prop:cb:uv}).
Moreover, again by Proposition~\ref{prop:cb:uv} the variety $X'$ is Gorenstein and so $E'$ is.
Let $C'\subset X'$ be the flipped curve. Then the birational map $\chi_E:E\dashrightarrow E'$ does not contract divisors.
By Zariski's main theorem the map $\chi^{-1}_E:E'\dashrightarrow E=\PP(1,1,3)$
is a morphism.
Then the singularities of $E'$ must be rational, hence they are Du Val.
Since $(K_{\tilde X}+E)\cdot \tilde C>0$, we have $(K_{X'}+E')\cdot C'<0$.
Since $K_{X'}+E'$ is negative on the fibers of $f'$
and $\uprho(X'/Z)=2$, the divisor $-(K_{X'}+E')$ is ample over~$Z$.
By the adjunction $-K_{E'}$ is ample. 
Then the singularities of $E$ must be Du Val, a contradiction.

Therefore, $E'$ is not normal and $\Gamma\subset \Delta_{f'}$. 
In this case $\Delta_{f'}$ is not a normal crossing curve at $o_2$. This implies that $X'$ has a singular point on $f'^{-1}(o_2)$.
If $P$ is the only singular point of~$X$, then
$\Delta_2'$ is a smooth curve and $\mult_{o_2}(\Delta_{f'})=2$  
by Claims~\ref{claimIA:sing} and~\ref{claimIA:discr}.
Hence, 
$X'$ is Gorenstein
by Proposition~\ref{prop:cb:uv}\ref{prop:cb:uv1}.
\end{proof}

\section{$\QQ$-conic bundles of type \typeb{ID^\vee}}
\label{Sect:ID}

In this section we adopt the following setup.

\begin{setup}
Let $f: X\to Z$ be a $\QQ$-conic bundle, which is of type \typek{ID}{} over a point $o\in Z$
and is a standard conic bundle outside~$o$. Let $C:=f^{-1}(o)_{\mathrm{red}}$ be the central fiber.
Recall that in this situation
the tubular neighborhood $(X\supset C)$ is biholomorphic to the quotient of the hypersurface
\[
\{ x_1^2+x_2^2+\theta(u,v) x_3^2=0\}\subset\PP^2_{x_1,x_2,x_3}\times\CC^2_{u,v}
\]
by the $\mumu_2$-action with weights
\[
\wt(x_1,x_2,x_3;\, u,v)=(0, 1,0;\,1,1),
\]
where $\theta\in \mathfrak{m}^2\subset \CC\{u,\, v\}$ is a $\mumu_2$-invariant. Recall that
$(X\supset C)$ is said to be \typek{ID}{m} if the rank of 
the quadratic part $\theta_{(2)}$ of $\theta$ equals $m$.

The variety~$X$ has a unique singular point, say~$P$, which is of type \typee{cA}2 in the cases \typek{ID}{2} and \typek{ID}{1},
and is of type \typee{cAx}2 in the case \typek{ID}{0}.
The point $o\in Z$ is of type \typeDu{A}1, hence 
$-K_X\cdot C=1$ (see \cite[Corollary~4.7.4]{MP:1pt}).
\end{setup}

\begin{sclaim}[Local description]
Near~$P$ we have the following identification
\begin{equation}
\label{eq:cDvee:eq}
(C\subset X)=\left(\{u=v=x_1^2+x_2^2=0\}\subset\{x_1^2+x_2^2+\theta(u,v)=0\}\subset \CC^4\right) /\mumu_2(0,1,1,1).
\end{equation}
\end{sclaim}

\begin{sclaim}
For the discriminant curve we have
\begin{equation}
\label{eq:IDvee:discrim}
\Delta_f=\{\theta(u,v)=0\}/\mumu_2.
\end{equation}
\end{sclaim}

\begin{theorem}
\label{thm:IDvee}
Let $f: X\to Z$ be a $\QQ$-conic bundle, which is of type \typek{ID}{} over a point $o\in Z$
and is a standard conic bundle outside~$o$.
Let $p: \tilde X\to X$ be a Kawamata blowup of the unique singular point $P\in X$.
Then $p$ can be completed to an md-link \eqref{eq:link:0},
where 
$\chi$ is the flip in the proper transform of the central fiber $C:=f^{-1}(o)_{\mathrm{red}}$,
$\al$ is the minimal resolution of $o\in Z$, and
$f'$ is a Gorenstein $\QQ$-conic bundle. For the discriminant curve $\Delta_{f'}$ we have
\begin{itemize}
\item 
$\Delta_{f'}$ is the set-theoretic preimage of $\Delta_f$ in the case \typek{ID}2
and in the case \typek{ID}0 with $\dif(X)=1$;
\item 
$\Delta_{f'}$ is the proper transform of $\Delta_f$ in the case \typek{ID}1 
and in the case \typek{ID}0  with $\dif(X)=2$.
\end{itemize}
\end{theorem}

\begin{proof}
We consider the cases \typek{ID}2, \typek{ID}1, \typek{ID}{0} separately.

\subsection*{Case \typek{ID}2}
Then up to analytic coordinate change we may assume that $\theta(u,v)=u^2+v^2$. 
According to \cite{Hayakawa:blowup1} in this case the only choice for the Kawamata blowup of $P\in X$ is the weighted blowup
with weight
\[\textstyle
w:=\frac12 (2,1,1,1).
\]
The variety
$\tilde{X}$ has a unique singularity, say $\tilde{P}$, which is a cyclic quotient of index $2$.
One can show that $\tilde{C}$ passes through $\tilde{P}$ and $E\cdot \tilde{C}=1$.
Then by \eqref{equation-computation-K.C-2} we have 
\[\textstyle
K_{\tilde{X}}\cdot \tilde{C}=K_X\cdot C+\frac12 E\cdot \tilde{C}=-\frac 12.
\]
Hence $\tilde{C}$ is a flipping curve and the link \eqref{eq:link:0} exists by 
Proposition~\ref{prop:link}.
Moreover, by \cite[Theorem~4.2]{KM:92} the variety $X'$ is smooth and $\chi: 
\tilde{X}\dashrightarrow X'$ is a single flip.
Hence $f'$ is a standard conic bundle. 
It remains to show that $\Delta_{f'}=\al^*\Delta_{f}$, i.e. $\Gamma\subset \Delta_{f'}$.
Assume the converse. Then $\Gamma$ meets $\Delta_{f'}$ transversely at two points.
It this case 
the surface $E'=f^{-1}(\Gamma)$ must be smooth and $\uprho (E')=4$.
On the other hand, the birational map  $\chi_E: E \dashrightarrow E'$ is 
an isomorphism on the subsets $E\setminus \tilde P$  and $E'\setminus C'$, 
where $C'$ is the flipped curve. Hence, $\uprho (E')\le \uprho (E)+1=2$, a contradiction.

\subsection*{Case \typek{ID}1}
Then  we may assume that $\theta(u,v)=v^{2k}-u^2$, where $k\ge 2$.
Apply to \eqref{eq:cDvee:eq} the following coordinate change:
\[
x= x_1+u,\quad y=x_1-u, \quad t= x_2.
\]
Then \eqref{eq:cDvee:eq} can be rewritten as follows

\[
(C \subset X)=\big(0\in \{x-y=v=0\}\subset\{xy+v^{2k}+t^2=0\}\subset \CC^4\big) /\mumu_2(1,1,1,0).
\]
Consider the weighted blowup
$p: \tilde{X}\to X$ of the point $P\in X$
with weight
\[\textstyle
w:=\frac12 (1,3,1,2).
\]
According to \cite{Kawamata-1992-e-app} or \cite{Hayakawa:blowup1} \
$p$ is a 
Kawamata blowup.\footnote{In this case 
there are exactly two choices for the extraction $p$: it can be weighted blowup with weights either
$\frac12 (1,3,1,2)$ or $\frac12 (3,1,1,2)$ \cite{Hayakawa:blowup1}.} 
The $p$-exceptional divisor $E$ is given in $\PP(1,3,1,2)$ by the equation 
\[
xy+\updelta(k,2) v^{2k}+t^2=0, 
\]
where $\updelta$ is the Kronecker delta.
Then $E$ is covered by the following affine charts:
\[
\renewcommand{\arraystretch}{1.5}
\begin{array}{rcl}
U_{x}&\simeq& \{\tilde y+\tilde v^{2k}\tilde x^{k-2}+\tilde t^2=0 \},
\\
U_{y}&\simeq& \{\tilde x+\tilde v^{2k}\tilde y^{k-2}+\tilde t^2=0\}
/\mumu_3(1,1,1,2),
\\
U_{v}&\simeq& \{\tilde x\tilde y+\tilde v^{k-2}+\tilde t^2=0\}.
\end{array}
\]
Thus 
$\tilde{X}$ has a unique non-Gorenstein point $\tilde{P}$ that is a cyclic quotient of index $3$
and has at worst isolated \type{cA} singularities outside~$\tilde{P}$.
In particular, $\dif(\tilde X)=2$.
The curve $\tilde C$ passes through $\tilde{P}$ and 
in the corresponding affine chart $U_{y}\simeq \CC^3_{\tilde y,\tilde v,\tilde t}/\mumu_3(1,1,2)$ we have 
\[
E=\{\tilde y=0\}/\mumu_3,
\quad 
\tilde C= \{\tilde v=\tilde y + \tilde t^2 =0\}/\mumu_3.
\]
Therefore, $E\cdot \tilde C=2/3$ and $K_{\tilde{X}}\cdot \tilde C=-2/3$
by \eqref{equation-computation-K.C-2}.
In particular, $-K_{\tilde X}$ is ample over~$Z$ and so
$\tilde C$ is a flipping curve. By Proposition~\ref{prop:flip:smooth} the flipped variety is Gorenstein. 
Thus we obtain an md-type link,
where 
$\chi$ is a single flip along $\tilde C$. 
Here $\al$ is a crepant contraction, so $Z'$ is smooth and 
$\al: Z'\to Z$ is the minimal resolution.
Moreover, by \eqref{eq:IDvee:discrim} the proper transform $\Delta'\subset Z'$ of $\Delta_f$ meets $\Gamma$ at a single point, say $o'$. By Corollary~\ref{cor:cb:plt4}
the curve $\Gamma$ is not contained in $\Delta_{f'}$ and
by Proposition~\ref{prop:cb:uv} the variety $X'$ is smooth outside $f'^{-1}(o')$.

\subsection*{Case \typek{ID}{0}}
Then $\mult_0(\theta)=2k$ with $k\ge 2$. Thus the equation of $X$ in \eqref{eq:cDvee:eq} admits an expansion 
\[
x_1^2+x_2^2+ \theta(u,v)= x_1^2+x_2^2+\theta_{(2k)}(u,v)+ \theta_{(2k+2)}(u,v)+\cdots=0,
\]
where $\theta_{(m)}$, as usual, denotes the homogeneous part of degree $m$.
We assume that $k$ is even. The case where $k$ is odd is treated in a similar way.
Following \cite[\S~8]{Hayakawa:blowup1} we distinguish two subcases:
\begin{itemize}
\item[\typek{ID'}{0}] 
the polynomial $\theta_{(2k)}$  is not a square and then $\dif(X)=1$,
\item [\typek{ID''}{0}] 
the polynomial $\theta_{(2k)}$  is  a square and then $\dif(X)=2$.
\end{itemize}
Consider these possibilities separately. 

\subsubsection*{Subcase: \typek{ID'}{0}, i.e $\theta_{(2k)}$  is not a square} 
Then there is only one choice for 
a Kawamata blowup of $P$: according to 
\cite[\S~8]{Hayakawa:blowup1} it must be the weighted blowup with weight
\[
\textstyle
w:=\frac12 (k,k+1,1,1).
\]
Clearly, the exceptional divisor $E$ is isomorphic to the hypersurface
\begin{equation}
\label{eq:hyp-curve}
\{x_1^2 + \theta_{(2k)}(u,v)=0\}\subset \PP(k,k+1,1,1).
\end{equation} 
It is covered by the following affine charts:
\[
\renewcommand{\arraystretch}{1.5}
\begin{array}{rcl}
U_{x_2}&\simeq& \{\tilde x_1^2+\tilde x_2+ \theta_{(2k)}(\tilde u,\tilde v)+ \tilde x_2 \theta_{(2k+2)}(\tilde u,\tilde v)+\cdots=0\}
/\mumu_{k+1}(k,-2,1,1),
\\
U_{u}&\simeq& \{\tilde x_1^2+\tilde x_2^2 \tilde u+ \theta_{(2k)}(1,\tilde v)+ \tilde u \theta_{(2k+2)}(1,\tilde v)+\cdots=0\},
\\
U_{v}&\simeq& \{\tilde x_1^2+\tilde x_2^2 \tilde v+ \theta_{(2k)}(\tilde u,1)+ \tilde v \theta_{(2k+2)}(\tilde u,1)+\cdots=0\}.
\end{array}
\]
The variety
$\tilde{X}$ has a unique non-Gorenstein point $\tilde{P}$ that is a cyclic quotient of type $\frac 1{k+1}(k,1,1)$
and has at worst isolated \type{cDV} singularities outside~$\tilde{P}$.
The curve $\tilde C$ passes through $\tilde{P}$ and in the corresponding affine chart $U_{x_2}\simeq \CC^3_{\tilde x_1,\tilde u,\tilde v}/\mumu_{k+1}(k,1,1)$ we have 
\[
E=\{\tilde x_1^2+ \theta_{(2k)}(\tilde u,\tilde v)+ \tilde x_2 \theta_{(2k+2)}(\tilde u,\tilde v)+\cdots=0\}/\mumu_{k+1},
\quad 
\tilde C= \{\tilde u=\tilde v=0\}/\mumu_{k+1}.
\]
Therefore, $E\cdot \tilde C=2/(k+1)$ and $K_{\tilde{X}}\cdot \tilde C=-k/(k+1)$
by \eqref{equation-computation-K.C-2}.
Hence $-K_{\tilde X}$ is ample over~$Z$ and 
$\tilde C$ is a flipping curve. By Proposition~\ref{prop:flip:smooth} the flipped variety is Gorenstein.
As in the case \typee{cA}2 we obtain an md-type link,
where 
$\chi$ is a single flip along $\tilde C$ and 
$\al: Z'\to Z$ is the minimal resolution (see Proposition~\ref{prop:flip:smooth}). 

Now we examine the surface $E'$.
It follows from \eqref{eq:hyp-curve} that the restriction $\varphi: E \dashrightarrow \PP(k,1,1)$ of the projection $\PP(k,k+1,1,1) \dashrightarrow \PP(k,1,1)$ is a rational map whose fibers 
$L_{\lambda}$ are 
rational curves and the image is the curve $\Upsilon\subset \PP(k,1,1)$ given by 
the same equation \eqref{eq:hyp-curve}. Moreover, 
$E$ is a weighted cone over $\Upsilon$. The curve $\Upsilon$ is smooth and hyperelliptic if 
$\theta_{(2k)}$ has no multiple factors and it is singular otherwise. Let $L_{\lambda}'\subset X'$ be the proper transform of $L_{\lambda}$. 
We have 
\[
E\cdot L_{\lambda}=
(\OOO(2) \cdot L_{\lambda})_{\PP(k,k+1,1,1)}=-1/(k+1). 
\]
Since $E'$ is a Cartier divisor the number $E'\cdot L_{\lambda}'$ must be integral.
Hence, $E'\cdot L_{\lambda}'=0$ by Lemma~\ref{lemma:flip} and because $-E'$ is nef.
This means that the proper transforms of fibers of $\varphi$ are contained in the fibers of $E'\to \Gamma$.
Now consider the curve $\Lambda:=\{x_2=x_1u\}\cap E$ and let $\Lambda'\subset E'$ be its proper transform.
The projection $\varphi: E \dashrightarrow \PP(k,1,1)$ induces an isomorphism $\Lambda\simeq \Upsilon$.
Since $\Lambda$ does not pass through $\tilde P$, we have an isomorphism $\Lambda\simeq \Lambda'$.
We claim that $E'$ is smooth at a general point of the flipped curve $C'$. Indeed, otherwise $f'_{E'}: E'\to \Gamma$ is a generically 
$\PP^1$-bundle and $\Lambda'$ is its section. In this case, $\Lambda'\simeq \Gamma$ is a smooth rational curve, 
and so $\Upsilon$ is also smooth. But then $\Upsilon$ must be a hyperelliptic (or elliptic) curve by \eqref{eq:hyp-curve}, a contradiction. 
Therefore, $E'$ is singular along the $f'$-horizontal curve $C'$ and so $\Gamma\subset \Delta_{f'}$. 

\subsubsection*{Subcase:  \typek{ID''}{0}, i.e  $ \theta_{(2k)}$ is a square}
Then there are two choices for Kawamata blowup \cite[\S~8]{Hayakawa:blowup1}. Since they are symmetric, we consider only one of them. 
Write $\theta_{(2k)}(u,v)=-\xi(u,v)^2$, where $\xi$ is a homogeneous polynomial of degree $k$.
Apply the following coordinate change to \eqref{eq:cDvee:eq}:
\[
y_1=x_1-\xi(u,v). 
\]
Then near $P$ the analytic neighborhood  $(C\subset X)$ has the following form
\[
\big(\{u=v=y_1^2+x_2^2=0\}\subset\{y_1^2+x_2^2+2y_1\xi(u,v) + \theta_{(2k+2)}(u,v)+\cdots =0\}\subset \CC^4\big) /\mumu_2(1,0,1,1).
\]
As above, consider the weighted blowup 
$p: \tilde{X}\to X$ of~$P$ with weight
\[
\textstyle
w=\frac12 (k+2,k+1,1,1).
\]
The exceptional divisor $E$ is isomorphic to the hypersurface
\begin{equation}
\label{eq:hyp-curve:prime}
\{x_2^2 + 2y_1\xi(u,v)+ \theta_{(2k+2)}(u,v)=0\}\subset \PP(k+2,k+1,1,1).
\end{equation} 
It is covered by the following affine charts:
\[
\renewcommand{\arraystretch}{1.5}
\begin{array}{rcl}
U_{y_1}&\simeq& \{\tilde y_1+\tilde x_2^2+2\xi(\tilde u,\tilde v)+ \theta_{(2k+2)}(\tilde u,\tilde v)+\cdots=0\}
/\mumu_{k+2}(-2,k+1,1,1),
\\
U_{u}&\simeq& \{\tilde y_1^2\tilde u+\tilde x_2^2 +2\tilde y_1\xi(1,\tilde v)+ \tilde u \theta_{(2k+2)}(1,\tilde v)+\cdots=0\},
\\
U_{v}&\simeq& \{\tilde y_1^2 \tilde v+\tilde x_2^2+2\tilde y_1\xi(\tilde u,1)+ \tilde v \theta_{(2k+2)}(\tilde u,1)+\cdots=0\}.
\end{array}
\]
As above,
$\tilde{X}$ has a unique non-Gorenstein point $\tilde{P}$ that is a cyclic quotient of type $\frac 1{k+2}(k+1,1,1)$
and has at worst isolated \type{cDV} singularities outside~$\tilde{P}$.
The curve $\tilde C$ passes through $\tilde{P}$ and
in the corresponding affine chart $U_{x_2}\simeq \CC^3_{\tilde y_1,\tilde u,\tilde v}/\mumu_{k+1}(k,1,1)$ we have 
\[
E=\{\tilde x_2^2+2y_1\xi(\tilde u,\tilde v)+ \theta_{(2k+2)}(\tilde u,\tilde v)+\cdots=0\}/\mumu_{k+2},
\quad 
\tilde C= \{\tilde u=\tilde v=0\}/\mumu_{k+2}.
\]
Therefore, $E\cdot \tilde C=2/(k+2)$ and $K_{\tilde{X}}\cdot \tilde C=-(k+1)/(k+2)$
by \eqref{equation-computation-K.C-2}.
This implies that $-K_{\tilde X}$ is ample over~$Z$ and~$\tilde C$ is a flipping curve. 
As above we obtain an md-type link,
where 
$\chi$ is a single flip along $\tilde C$ and 
$\al: Z'\to Z$ is the minimal resolution (see Proposition~\ref{prop:flip:smooth}).
Now consider the pencil of curves 
\[
L_{\lambda}:=\{u=\lambda v\}\cap E
\]
and let $L_{\lambda}'\subset E'$ be their proper transforms.
We have 
\[
\textstyle
E\cdot L_{\lambda}=(\OOO(2) \cdot L_{\lambda})_{\PP(k+2,k+1,1,1)}=-\frac{4}{k+2}\ge -1. 
\]
Since $E'$ is a Cartier divisor the number $E'\cdot L_{\lambda}'$ must be integral.
Hence, $E'\cdot L_{\lambda}'=0$ by Lemma~\ref{lemma:flip} and because $-E'$ is nef.
This means that a general curve $L_{\lambda}'$ is contained in a fiber of $E'\to \Gamma$.
Now consider the curve 
\[
\Lambda:=\{y_1=\phi(u,v)\}\cap E, 
\]
where $\phi$ is a 
general homogeneous polynomial of degree $k+2$, and let $\Lambda'\subset E'$ be its proper transform.
Note that the intersection $L_{\lambda}\cap \Lambda$ is contained in the smooth part of $E$ and
$\Lambda$ does not pass through $\tilde P$ and so the intersection $L_{\lambda}'\cap \Lambda'$ is contained in the smooth part of $E'$. Then the map $\chi$ is an isomorphism near $\Lambda$.
Hence 
\[
f'^*\Gamma\cdot \Lambda'=E'\cdot \Lambda'=E\cdot \Lambda=(\OOO(2) \cdot \Lambda)_{\PP(k+2,k+1,1,1)}= -4
\]
Since  $\Gamma^2=-2$, the restriction $f'_{\Lambda'}:\Lambda'\to \Gamma$ is 
finite of degree $2$.
On the other hand, $(L_{\lambda}'\cdot \Lambda')_{E'}=(L_{\lambda}\cdot \Lambda)_{E}=2$, 
hence $L_{\lambda}'$ is a whole fiber of $E'/\Gamma$.
This means that a general geometric fiber of $E'/\Gamma$ is irreducible.
Therefore, $E'$ is normal and $\Gamma\not\subset \Delta_{f'}$. 
\end{proof}


\end{document}